\documentclass[12pt,a4paper,reqno]{amsart}

\usepackage[active]{srcltx}
\usepackage{enumitem}
\usepackage{float}
\usepackage{amssymb}

\usepackage{lineno,hyperref}
\usepackage[active]{srcltx}
\usepackage{float}
\usepackage{mathtools}
\usepackage[all]{xy}

\usepackage{xcolor}

\usepackage{tikz}
\usetikzlibrary{arrows,decorations.markings,shapes.geometric}
\tikzstyle{point}=[circle, inner sep = 2pt, outer sep = 1pt, minimum size = 5pt, fill=black, draw=black]

\newcommand{\Aut}{\operatorname{Aut}}
\newcommand{\Hom}{\operatorname{Hom}}
\newcommand{\Z}{{\mathbb Z}}
\newcommand{\M}{{\mathcal M}}
\newcommand{\Q}{{\mathbb Q}}

\newcommand{\G}{\mathcal{G}}
\newcommand{\E}{{\mathcal E}}

\newcommand{\Arr}{\operatorname{Arr}}

\newcommand{\s}{\mathcal{S}}
\newcommand{\Graphs}{\mathcal{G}raphs}

\newcommand{\Cay}{\operatorname{Cay}}
\newcommand{\IRel}{{I\mathcal{R}el}}
\newcommand{\CDGA}{\operatorname{CDGA}}
\newcommand{\Dig}{\operatorname{sDig}}

\newtheorem{theorem}{Theorem}[section]
\newtheorem{lemma}[theorem]{Lemma}
\newtheorem{corollary}[theorem]{Corollary}
\newtheorem{proposition}[theorem]{Proposition}

\newtheorem{question}[theorem]{Question}

\theoremstyle{definition}

	\theoremstyle{definition}
	\newtheorem{rem&notation}[theorem]{Remark and notation}

\theoremstyle{definition}
\newtheorem{remark}[theorem]{Remark}

\theoremstyle{definition}
\newtheorem{definition}[theorem]{Definition}

\numberwithin{equation}{section}

\title[Realisability problem in arrow categories]{Realisability problem in arrow categories}

\author{Cristina ~Costoya}
\address[C.~Costoya]{CITIC Research Center, Universidade da Coru{\~n}a, Departamento de Computaci\'on, Campus de Elvi{\~n}a, 15071  A Coru{\~n}a, Spain.}
\email[C.~Costoya]{cristina.costoya@udc.es}

\author{David ~M\'endez}
\address[D.~M\'endez]{
	Departamento de {\'A}lgebra, Geometr{\'\i}a y Topolog{\'\i}a,
	Universidad de M{\'a}\-la\-ga, Campus de Teatinos, 29071 M{\'a}laga,
	Spain.}
\email[A.~Viruel]{david.mendez@uma.es}

\author{Antonio ~Viruel}
\address[A.~Viruel]{
Departamento de {\'A}lgebra, Geometr{\'\i}a y Topolog{\'\i}a,
Universidad de M{\'a}\-la\-ga, Campus de Teatinos, 29071 M{\'a}laga,
Spain.}
\email[A.~Viruel]{viruel@uma.es}


\thanks{
First and second authors are partially  supported by Ministerio de Econom{\'\i}a y Competitividad (Spain), grant MTM2016-79661-P (AEI/FEDER, UE, support included). Second author is partially supported by Ministerio de Educaci\'on, Cultura y Deporte (Spain) grant FPU14/05137. Second and third authors are partially supported by  Ministerio de Econom{\'\i}a y Competitividad (Spain),  grant MTM2016-78647-P (AEI/FEDER, UE, support included). }

\begin{document}

\begin{abstract}
	In this paper we raise the realisability problem in arrow categories. Namely, for a fixed category $\mathcal{C}$ and for arbitrary groups $H\le G_1\times G_2$, is there an object $\phi \colon A_1 \rightarrow A_2$ in $\Arr(\mathcal{C})$ such that $\Aut_{\Arr(\mathcal{C})}(\phi) = H$, $\Aut_{\mathcal{C}}(A_1) = G_1$ and $\Aut_{\mathcal{C}}(A_2) = G_2$? We are interested in solving this problem when $\mathcal C =\mathcal{H}oTop_*$, the homotopy category of simply-connected pointed topological spaces. To that purpose, we first settle that question in the positive when $\mathcal C = \Graphs$.

	Then, we construct an almost fully faithful functor from $\Graphs$ to $\CDGA$, the category of commutative differential graded algebras, that provides among other things, a positive answer to our question when $\mathcal C = \CDGA$ and, as long as we work with finite groups, when $\mathcal C =\mathcal{H}oTop_*$. Some results on  representability of concrete categories are also obtained.
\end{abstract}

\maketitle

\section{Introduction}\label{section:introduction}
Let $\mathcal{C}$ be a category and $G$ be a group. The  classical group realisability problem asks for an object $X\in \mathcal{C}$ such that $\Aut_\mathcal{C}(X)\cong G$. In that case we say that the group $G$ is realised in the category $\mathcal C$. This problem has been treated in the literature in many  contexts.
In 1936, K\"onig  \cite{kon36} raised that question for $\mathcal C = \Graphs$, the category of simple, undirected, connected graphs, which was first solved by Frucht \cite{Fru39} in the case of finite groups,  and afterwards by  de Groot \cite{Gro59} in the general case. This question was also addressed  by Kahn in the 1960's for $\mathcal C = \mathcal{H}oTop_*$, the homotopy category of simply-connected pointed topological spaces, and it received a remarkable attention.  However, the tools for seriously digging into Kahn's question were at the time insufficient and this problem appeared recurrently in lists of open problems and surveys in homotopy theory \cite{A2, F2, ka1, ka3, Ru}. The impasse ended with a general method by the authors of this paper, \cite{CosVir14},  that gives a solution to the classical group realisability problem in  $\mathcal C = \mathcal{H}oTop_*$ for the case of finite groups. At the present, the general case remains unsolved. However, very recently have the same authors  tackled that problem from a completely different perspective  and, via Invariant Theory, have solved it for a large family of groups strictly containing finite ones, \cite{CosVir18}.

The goal of this paper is to address the realisability problem in the more general setting of arrow categories, especially in $\Arr(\mathcal{H}oTop_*)$. Recall that the arrow category of a category $\mathcal{C}$, denoted $\Arr(\mathcal{C})$, is the category whose objects are morphisms $f\in \Hom_\mathcal{C}(A,B)$, $A,B\in \mathcal{C}$, and where a morphism between two objects $f_i\colon A_i\to B_i \in \Hom_\mathcal{C}(A_i,B_i)$, $i=1,2$, is a pair of morphisms of $\mathcal{C}$, $(a, b)\in \Hom_\mathcal{C}(A_1, A_2)\times \Hom_\mathcal{C} (B_1, B_2)$, such that $b f_1 = f_2 a$. By abuse of notation, we denote $\Aut_\mathcal{C}(f) = \Aut_{\Arr(\mathcal{C})}(f)$. We ask the following:

\begin{question}\label{question:ArrowRealisation}(Realisability in arrow categories)
	Let  $G_1, G_2$ be groups, $H \le G_1 \times G_2$, and let $\mathcal{C}$ be a given category. Is there an object $f\colon A_1\to A_2$ in $\Arr(\mathcal{C})$ such that $\Aut_{\mathcal{C}}(f)\cong H$ and $\Aut_\mathcal{C}(A_i) \cong G_i$, $i = 1,2$?
\end{question}

Notice that this question implies a strong link between the existence of objects in $\Arr(\mathcal{C})$ realising groups, and the existence of objects in $\mathcal C$ realising groups. Thus, a good idea is to restrict ourselves to categories $\mathcal C$ for which we know that, if not for arbitrary groups, at least for finite groups a solution to the classical group realisability problem exists.  Let us start by considering $\mathcal C = \Graphs$ where, as we have mentioned at the beginning, any arbitrary group can be realised, \cite{Gro59}. Our main result settles positively Question \ref{question:ArrowRealisation} in $\Arr(\Graphs)$:

\begin{theorem}\label{theorem:arrowGraphRealisation}
	Let $G_1$, $G_2$ be groups and $H\le G_1\times G_2$.  There exist $\G_1$, $\G_2$ objects  in $\Graphs,$ and $\varphi\colon \G_1\to \G_2$ object in $\Arr(\Graphs)$, such that $\Aut_{\Graphs}(\varphi)\cong H$ and $\Aut_{\Graphs}(\G_i) \cong G_i$, for $i=1,2$.
\end{theorem}

On a technical note we point out that to prove Theorem \ref{theorem:arrowGraphRealisation} we adopt the same strategy that Frucht \cite{Fru39} and de Groot \cite{Gro59} followed to solve the classical group realizability problem in $\mathcal C = \Graphs$.  Namely, we first answer positively Question \ref{question:ArrowRealisation} in the categorical framework of binary relational systems (see Theorem \ref{theorem:MainArrowRealization}). Then, by arrow replacement (see Section \ref{section:toGraphs}) we transform those relational systems into graphs  that will serve our purpose.

\smallskip

To answer Question \ref{question:ArrowRealisation}  in different arrow categories, other than $\Arr(\Graphs)$, our idea is to transfer our previous result, Theorem  \ref{theorem:arrowGraphRealisation}, from $\Arr(\Graphs)$ to these arrow categories.
Fundamental to us is the existence of functors with domain $ \Graphs $ that preserve automorphism groups. In this paper we construct, for each integer $n \geq 1,$ an \emph{almost} fully faithful functor $\mathcal M_n$ between the intermediate categories, $\Dig$ (strongly connected digraphs without loops), and the category of  $n$-connected commutative differential graded algebras (see Theorem \ref{theorem:computingMorphisms}). When $\mathcal M_n$ is restricted to the full subcategory  $\Graphs$, it preserves automorphism groups and therefore a complete realisability result is  obtained in the arrow category of commutative differential graded algebras, $\Arr( \CDGA)$:

\begin{theorem}\label{theorem:CDGArealisation}
Let $G_1$, $G_2$ be groups and $H\le G_1\times G_2$. For any $n\ge 1$, there exist $M_1$ and $M_2$  ($n$-connected) objects in $\CDGA,$ and $\varphi\colon M_1\to M_2$, an object in  $\Arr( \CDGA)$, such that $\Aut_{\CDGA}(\varphi) \cong H$ and $\Aut_{\CDGA}(M_i)\cong G_i$, $i=1,2$.
\end{theorem}

We emphasise that Theorem \ref{theorem:CDGArealisation} has strong consequences. It solves, in the positive, the classical group realisability problem in the category of commutative differential graded algebras. That is, every group $G$ (even infinite ones!) can be realised in  $\mathcal C =\CDGA$. This is much of an improvement on our previous result \cite[Theorem 4]{CV3} where locally finite groups were realised but only up to some extension.

In this picture of transferring results among categories,  let us recall that finiteness conditions are  key to translate properties from $\mathcal C = \CDGA$  to  $\mathcal C = \mathcal{H}oTop_*$  via Sullivan's spatial realisation functor \cite[Chapter 17]{FelHalTho01}.
Even though the graphs of Theorem \ref{theorem:arrowGraphRealisation} are not necessarily finite, and therefore the algebras in Theorem \ref{theorem:CDGArealisation} are not of finite type,  finite graphs can be chosen  if the groups involved are so (see Corollary \ref{corollary:finiteGraphs}). Hence, we obtain the following realisability result  in $\Arr(\mathcal{H}oTop_*)$:

\begin{corollary}\label{theorem:HoTopRealisation}
	Let $G_1$, $G_2$ be finite groups and let $H \leq G_1\times G_2$. For every $n\ge 1$, there exist ($n$-connected) objects $X_1$, $X_2$ in $\mathcal{H}oTop_*$ and $f\colon X_1\to X_2$ object in $\Arr(\mathcal{H}oTop_*)$ such that $\E(X_i)\cong G_i$, $i=1,2$, and $\E(f)\cong H$.
\end{corollary}

Recall that in $\mathcal{H}oTop_*$ the group of automorphisms of an object $X$ is the group of homotopy classes of pointed self-homotopy equivalences, denoted by $\E(X)$. By analogy, we denote $\E(f) = \Aut_{\mathcal{H}oTop_*}(f)$.

Returning to our functor $\mathcal M_n$, $n \geq 1$, we provide some results on representability of categories in $\CDGA$ and $\mathcal{H}oTop_*$.
Recall that a category $\mathcal{C}$ is concrete if it admits a faithful functor $F\colon\mathcal{C}\to\operatorname{Set}$, and finite if it has finitely many objects and morphisms between any two objects. Using that $\mathcal M_n$ is \emph{almost} fully faithful as well as some known results on representability of categories \cite{HelNes04,PulTrn80}, we prove the following:

\begin{corollary}\label{theorem:categoryRepresentation}
	Let $\mathcal{C}$ be a concrete small category. For every $n\ge 1$, there exists a functor $G_n\colon\mathcal{C} \to \CDGA$ such that $$\Hom_{\CDGA}\big(G_n(A), G_n(B)\big)\setminus\{0\} = [G_n(A), G_n(B)]\setminus\{0\} = \Hom_{\mathcal{C}}(A,B),$$ for any $A,B\in\mathcal{C}$. Moreover, if $\mathcal{C}$ is a finite category, for every $n\ge 1$ there exists a functor $F_n\colon \mathcal{C}\to \mathcal{H}o{Top}_*$ such that $$[F_n(A),F_n(B)]\setminus\{[0]\} = \Hom_{\mathcal{C}}(A,B),$$ for any $A,B\in\mathcal{C}$.
\end{corollary}	

Now recall that any monoid $M$ can be regarded as a single object category. Such a category is always concrete, and it is finite whenever $M$ is so. Also notice that if $M$ has a zero element, it becomes a zero endomorphism of the only object in the category associated to $M$. The following becomes immediate:

\begin{corollary}\label{corollary:finiteMonoids}
	Let $M$ be a monoid. For every $n\ge 1$, there exists an ($n$-connected) $\CDGA$ $M_n$ such that $\Hom_{\CDGA}(M_n, M_n) = [M_n, M_n] \cong M^0$. Moreover, if $M$ is finite, there exists a ($n$-connected) space $X_n$ such that $[X_n,X_n]\cong M^0$.
\end{corollary}

Here, $M^0$ denotes the monoid obtained from $M$ by adjoining a zero element. In particular, if $M\cong N^0$ for some other monoid $N$, that is, if $M$ has a zero element and no non-trivial zero divisors, we can realise it directly.

\medskip

\emph{Outline of the paper.} A general answer to Question \ref{question:ArrowRealisation} in the framework of binary relational systems is given in Section \ref{section:realcolgraph}.  In Section \ref{section:toGraphs}, arrow replacement techniques applied to binary relational systems  provide us with a positive answer to  Question \ref{question:ArrowRealisation} for $\mathcal C =\Graphs$. We introduce in Section \ref{section:topSpaces} the family of functors $\mathcal M_n: \Dig \rightarrow \CDGA$, $n \geq 1$,  between strongly connected digraphs and commutative differential graded algebras. The rest of Section \ref{section:topSpaces} is dedicated to assemble those previous results to get a positive answer to Question \ref{question:ArrowRealisation} for $\mathcal C = \CDGA$ and, under finiteness conditions, for $\mathcal C = \mathcal{H}oTop_*$. At the end of Section \ref{section:topSpaces} we deduce some corollaries. Finally, in Section \ref{sec:example} we illustrate with an example the constructions we used to answer Question  \ref{question:ArrowRealisation}.

\smallskip
Extensive use of graph theory is made throughout this work, for which \cite{HelNes04, Whi84} are our main references. Some knowledge of rational homotopy theory is needed for Section \ref{section:topSpaces}, so we refer to \cite{FelHalTho01} for the basic facts on the subject. For a set $A$ containing the $0$, throughout this paper  $A^\ast = A \setminus \{0\}$.

\section{Realisability in the arrow category of Binary relational systems }\label{section:realcolgraph}
A convenient way to construct objects in $\Arr(\Graphs)$ that answer in the positive Question \ref{question:ArrowRealisation} is to first work in the categorical framework of binary relational systems that are now introduced following the notation in \cite{HelNes04}.

\begin{definition}
	A binary relational system $\s$ over a set $I$ consists in a set $ V(\s)$ together with a family of binary relations $R_i(\s)$ on $V(\s)$,  $i\in I$. They are also referred to as binary relational $I$-systems. When  $I$ has one element, $\s$ is a directed graph; otherwise $\s$ can be seen as a directed graph with labelled edges.  Elements of $V(\s)$ are called vertices and elements of $R_i(\s)$ edges of label $i$. A morphism $f\colon \s_1\to \s_2$ of binary relational systems over $I$ is a map $f\colon V(\s_1) \to V(\s_2)$ such that $\big(f(v),f(w)\big)\in R_i(\s_2)$, whenever $(v, w)\in R_i(\s_1)$,  $i \in I$. We will write $f \in \Hom_{\IRel} ( \s_1, \s_2)$. The group of automorphisms of a binary system $\s$ over a set $I$ is denoted by $\Aut_{\IRel}(\s)$. A full binary relational $I$-subsystem of $\s$ is a binary $I$-system $\s'$ such that $V(\s')\subseteq V(\s)$ and, for every $v,w\in V(\s')$ and $i\in I$, $(v,w)\in R_i(\s)$ if and only if $(v,w)\in R_i(\s')$.
\end{definition}

Our main result in this section is the following:
\begin{theorem}\label{theorem:MainArrowRealization}
	Let $G_1$ and $G_2$ be arbitrary groups and $H \leq G_1 \times G_2$. There exists a morphism of binary relational systems over a certain set $I$, $\varphi: \G_1\rightarrow \G_2$, such that  $\Aut_{\IRel}(\G_1)$, $\Aut_{\IRel}(\G_2)$, and $\Aut_{\IRel}(\varphi)$ are, respectively, isomorphic to $G_1$, $G_2$, and $H$.
\end{theorem}

The construction of the binary relational systems involved in Theorem \ref{theorem:MainArrowRealization} is carried out in Subsection \ref{subsec:relationalsystems}; properties of their automorphism groups are given in Subsection \ref{subsec:proprelsyst}; and everything is put together to prove Theorem \ref{theorem:MainArrowRealization} in Subsection \ref{subsec:proofmaintheorem}.

Fundamental to our constructions is the notion of a Cayley diagram that we recall now:

\begin{definition}\label{def:cayleyDiag}
	Let $G$ be a group and $S = \{s_i \mid i\in I\}$ a generating set of $G$.  The Cayley diagram of $G$ associated to $S$ is the binary $I$-system $\Cay(G,S)$ with $V\big(\Cay(G,S)\big) = G$ and, a pair $(g, g')\in R_i\big(\Cay(G,S)\big)$ if and only if $s_i g = g'$.
\end{definition}

\begin{remark}\label{remark:fixvertex}
	Recall from  \cite[Section 6]{Gro59} or \cite[Section 3.3]{CoxMos80} that, regardless the chosen generating set $S$, $\Aut_{I\mathcal{R}el}\big(\Cay(G,S)\big)\cong G$. An element $h\in G$ determines an automorphism of the Cayley diagram, which we will denote by $\phi_h $ in this paper. It corresponds to right multiplication by $h^{-1}$ on the vertices $G$, an action which preserves labelled edges. Hence if $\phi_h$ fixes any vertex then it is necessarily the identity. Furthermore, $\phi_{h^{-1}}$ is the only automorphism of $\Cay(G,S)$ sending $e_G$ to $h$.
\end{remark}

As a part of our work we need a characterisation of the subgroups of a product of two groups. An elementary result, known as Goursat's lemma,  is used to that purpose. The basic idea of the lemma's proof can be found in \cite[Theorem 2.1 and p.\ 3]{BauerSenZven}.
\begin{lemma}[{\cite[Sections 11--12]{Gou89}}]\label{lemma:goursat}
	Let $G_1$ and $G_2$ be arbitrary groups and $H\le G_1\times G_2$. Consider $\iota_j\colon G_j\to G_1\times G_2$ and $\pi_j\colon G_1\times G_2\to G_j$ the respective inclusions and projections, $j=1,2$. There exists a group isomorphism
		\[\theta\colon \frac{\pi_1(H)}{\iota_1^{-1}(H)}\longrightarrow \frac{\pi_2(H)}{\iota_2^{-1}(H)},\]
	taking a class $[g_1]$ to $\theta([g_1]) = [g_2]$,
	the class of any element $g_2\in \pi_2(H)$ such that $(g_1,g_2)\in H$. Moreover,
		\[H = \big\{(g_1, g_2) \in \pi_1(H)\times \pi_2(H) \mid \theta([g_1]) = [g_2]\big\}.\]
\end{lemma}

\medskip

\subsection{Construction of the binary relational systems involved in Theorem \ref{theorem:MainArrowRealization}}\label{subsec:relationalsystems} Taking into account the previous lemma, we now proceed with the construction of the binary relational systems in Theorem \ref{theorem:MainArrowRealization}. Let $G_1$ and $G_2$ be arbitrary groups and $H \leq G_1 \times G_2$.

\begin{definition}\label{def:decompgroups}{(Generating sets $R$ and $S$ for, respectively, $G_1$ and $G_2$)}
	\begin{enumerate}
		\item Let $J_1 $ be an indexing set for the right cosets of $\iota_1^{-1}(H)$  in $G_1$.   We choose a representative of each right coset, $\{r_j, \, j \in J_1\}$, assuming that $0 \in J_1$ and $r_0 =e_{G_1}$ represents $\iota_1^{-1}(H)$. We fix a generating set  $\{r_i\mid i\in I_{\iota_1}\}$ for $\iota_1^{-1}(H)$ and we let $I_1 = I_{\iota_1}\sqcup J_1^\ast$. Then $R = \{r_i\mid i\in I_1 \}$ is a generating set for $G_1$.

		\item Let $J_2$ be an indexing set for the right cosets of $\pi_2(H)$ in $G_2$.  Analogously, we choose a representative of right cosets $\{s_j\mid j\in J_2\}$, assuming that $0 \in J_2$ and $s_0=e_{G_2}$ represents $\pi_2(H)$. We fix a generating set  $\{s_i\mid i\in I_{\pi_2}\}$ for $\pi_2(H)$ and we let  $I_2 = I_{\pi_2}\sqcup J_2^\ast$. Then $S = \{s_i\mid i\in I_2\}$ is a generating set for $G_2$.
	\end{enumerate}
\end{definition}

\begin{remark}\label{rem:decompgroups}
	By decomposing $G_1=\sqcup_{j\in J_1} \iota_1^{-1}(H)r_j$, there exist maps $k_1\colon G_1\to \iota_1^{-1}(H)$ and $j_1\colon G_1\to J_1$ such that any $g\in G_1$ can be uniquely expressed as a product $g = k_1(g) r_{j_1(g)}$. By setting $J_{\pi_1} = \{ j \in J_1 \mid r_j\in \pi_1(H)\}$, if $g\in \pi_1(H)$ we have that $j_1(g)\in J_{\pi_1}.$ Analogously,  $G_2 = \sqcup_{j\in J_2} \pi_2(H) s_j$ and there exist maps $k_2\colon G_2\to \pi_2(H)$ and $j_2\colon G_2\to J_2$ such that any $g\in G_2$ is uniquely expressed as the product $g = k_2(g) s_{j_2(g)}$.
\end{remark}

The maps $k_1, j_1, k_2$ and $j_2$ satisfy certain compatibility conditions with the group operation:

\begin{lemma}\label{lemma:PropjAndk}
	Let $g, g'\in G_1$ (resp. $g, g'\in G_2$). Then,
	\begin{enumerate}
		\item $j_1(gg') = j_1(r_{j_1(g)} g')$ (resp. $j_2(g g') = j_2(s_{j_2(g)} g')$).
		\item $k_1(gg') = k_1(g) k_1(r_{j_1(g)}g')$ (resp.  $k_2(gg') = k_2(g) k_2(s_{j_2(g)}g')$).
	\end{enumerate}
\end{lemma}

\begin{proof}
	We present the proof for $G_1$ only since the one for  $G_2$ is analogous.  By Remark \ref{rem:decompgroups}, $g g' = k_1(g g') r_{j_1(g g')}$ and also $g g' = k_1(g) r_{j_1(g)} g' = k_1(g) k_1(r_{j_1(g)}g') r_{j_1(r_{j_1(g)}g')}$. Since this decomposition is unique, \emph{(1)} and \emph{(2)}  follow immediately.	
\end{proof}

The following is an auxiliary binary system that will be used in Definition \ref{def:G1andG2}.

\begin{definition}(Auxiliary binary system)\label{def:auxsystem}
	Let	$I = I_1  \sqcup I_2 \sqcup \{\theta\}$ and $V_1 = \pi_1(H)/\iota_1^{-1}(H)  $. We define $\G_{\iota_1}$ to be the binary $I$-system having as vertices
		\[V(\G_{\iota_1}) =
			\begin{cases}
				V_1, & \text{if } G_1 = \pi_1(H),  \\
				V_1  \sqcup \{s\},    & \text{otherwise},
			\end{cases}	\]	
	and as edges of label $i$, with $[g] \in V_1$
		\begin{enumerate}
			\item for $ i\in I_{\iota_1}$,  $([g], [g])  \in R_i(\G_{\iota_1})$;
			\item for $i\in J_{\pi_1}^\ast$, 	$([g], [r_ig]) \in R_i(\G_{\iota_1}).$
		\end{enumerate}
	If $G_1 \ne  \pi_1(H)$, $\G_{\iota_1}$ also has the edges of label $i$
		\begin{enumerate}
			\item[(3)]  for $ i\in I_1$, $(s,s ) \in R_i(\G_{\iota_1}) $;
			\item[(4)] for $i\in J_1 \setminus J_{\pi_1}$, $	\big([g], s\big), \big(s, [g]\big) \in R_i(\G_{\iota_1}) $.
		\end{enumerate}
	Observe that the set of edges of $\G_{\iota_1}$ corresponding to labels in $I_2\sqcup \{\theta\}$ is empty.
\end{definition}

\begin{remark}\label{fullsubGi1}
	The Cayley diagram $\Cay\big(V_1,\{[r_i] \mid i\in I_{\iota_1}\sqcup J_{\pi_1}^\ast\}\big)$ is equal to $\G_{\iota_1} $ if $G_1 = \pi_1 (H)$ and is a proper full binary relational subsystem otherwise.
\end{remark}

We are now ready to define the binary $I$-systems $\G_1$ and $\G_2$ in Theorem \ref{theorem:MainArrowRealization}. Recall from Lemma \ref{lemma:goursat} that there exists an isomorphism $\theta\colon\pi_1(H)/\iota_1^{-1}(H) \to \pi_2(H)/\iota_2^{-1}(H)$. Also recall that in Definition \ref{def:decompgroups} we described two generating sets $R $ and $S$ for, respectively, $G_1$ and $G_2$:

\begin{definition}\label{def:G1andG2}(Binary relational systems $\G_1$ and $\G_2$ in Theorem \ref{theorem:MainArrowRealization})
	We define the following binary $I$-systems:
	\begin{enumerate}
		\item $\G_1 = \Cay(G_1,R)$.
		\item $\G_2$ has vertex set $V(\G_2) = G_2\sqcup \big(\sqcup_{j\in J_2} V_2^{j}\big)$  where $V_2^{j} = \{j\}\times V(\G_{\iota_1})$, and edge set:
		\begin{enumerate}
			\item[--] for $i\in I_2$ and $g\in G_2$, $(g, s_i g)\in R_i(\G_2)$;
			\item[--] for $\theta \in I$ and $g\in G_2$, $\big(g,\big(j_2(g),\theta^{-1}[k_2(g)]\big)\big)\in R_\theta(\G_2)$;
			\item[--] for $i\in I_1$, $j\in J_2$,  if $(v_1, v_2)\in R_i(\G_{\iota_1})$, then $\big((j, v_1), (j, v_2)\big) \in R_i(\G_2)$.
		\end{enumerate}
	\end{enumerate}
\end{definition}

\begin{remark}\label{rem:fullsub}
	Cases of interest for us in the paper are the following full binary relational subsystems of $\G_2$:
	\begin{enumerate}
		\item $\G_2(G_2)$, with vertex set $G_2$, which is  isomorphic to $\Cay(G_2,S)$.
		\item $\G_2(G_1,j)$, with vertex set $V_2^{j}$,  which is isomorphic to $\G_{\iota_1}$ for each $j\in J_2$.
	\end{enumerate}
\end{remark}

The following is an auxiliary construction that will be used in Definition \ref{def:phi}:

\begin{lemma} (Auxiliary morphism of binary systems) 
	The map $\varphi_0\colon \G_1\to \G_{\iota_1}$ defined by
		\[\varphi_0(g) =
		\begin{cases}
			[g],& \text{if } g\in \pi_1(H),  \\
			s,    & \text{otherwise, }
		\end{cases}\]
	is a morphism of binary $I$-systems.
\end{lemma}

\begin{proof}
	We need to check that for $i \in I_1$, if $(g, r_i g)\in R_i(\G_1)$ then $\big(\varphi_0(g),\varphi_0(r_i g)\big) \in R_i(\G_{\iota_1})$. By decomposing $I_1 =  I_{\iota_1} \sqcup J_{\pi_1}^\ast \sqcup ( J_1 \setminus J_{\pi_1})$, we have the following:

	For $i\in I_{\iota_1}$, $r_i \in \iota_1^{-1}(H) $. Hence, if $g\in \pi_1(H)$, also $r_i g \in \pi_1(H) $ and  $\big(\varphi_0(g),\varphi_0(r_i g)\big) = \big([g], [r_ig]\big) = \big([g], [g]\big) \in R_i(\G_{\iota_1})$. On the other side, if $g\not\in \pi_1(H)$,  then $r_i g\not \in \pi_1(H)$ and by definition, $\big(\varphi_0(g),\varphi_0(r_i g)\big) = (s, s)\in R_i(\G_{i_1})$.

	For $ i\in J_{\pi_1}^\ast$, $r_i\in \pi_1(H)$ and the argument goes as previously.

	Finally, for $i\in J_1 \setminus J_{\pi_1}$, $r_i\not\in \pi_1(H)$. On the one hand, if $g\in \pi_1(H)$ necessarily  $r_i g\not \in \pi_1(H)$, therefore $\big(\varphi_0(g),\varphi_0(r_i g)\big) = \big([g], s\big) \in R_i(\G_{i_1})$. On the other hand, if $g\not\in \pi_1(H)$, we can have either $r_i g\in\pi_1(H)$, in which case $\big(\varphi_0(g),\varphi_0(r_i g)\big) = (s, [r_i g]) \in R_i(\G_{i_1})$, or $r_i g\not\in\pi_1(H)$, in which case $\big(\varphi_0(g),\varphi_0(r_i g)\big) = (s, s) \in R_i(\G_{i_1})$.
\end{proof}

\begin{definition}\label{def:phi}(Arrow $\varphi\colon\G_1\to \G_2$ in Theorem \ref{theorem:MainArrowRealization}) 
	Let $\varphi\colon\G_1\to \G_2$ be the composite of the morphism $\varphi_0$ from the previous lemma, followed by the inclusion of $\G_{\iota_1} \cong \G_2(G_1,0)$ (see Remark \ref{rem:fullsub} (2)) into $\G_2$:
	\begin{displaymath}
			\xymatrix{
			\G_1  \ar[r]_-{\varphi_0}
			& \G_{\iota_1} \cong \G_2(G_1,0)  \ar@{^{(}->}[r]_-{i_0} & \G_2.}
		\end{displaymath}
	That is, $\varphi(g) = \big(0, \varphi_0(g)\big) \in V_2^{0}$ for $g\in V(\G_1) = G_1$.
\end{definition}

For the sake of clarity, we split the proof of Theorem \ref{theorem:MainArrowRealization} into various intermediate results that we include in the following subsection.

\subsection{Properties of the binary relational systems from Definition \ref{def:G1andG2}}\label{subsec:proprelsyst}

Since $\G_1 $ is a Cayley diagram for $G_1$, we have that $\Aut_{\IRel}(\G_1)\cong G_1$ (see Remark \ref{remark:fixvertex}). Proving that $\Aut_{\IRel}(\G_2)\cong G_2$ needs further elaboration. The first step is to prove that $\G_{\iota_1}$, the auxiliary binary $I$-system introduced in Definition \ref{def:auxsystem}, is sufficiently rigid:

\begin{lemma}\label{lemma:autGi1}
	For a fixed $g\in \pi_1(H) $, there exists a unique  $\psi_g\in\Aut_{IRel}(\G_{\iota_1})$ such that $\psi_g([e_{G_1}]) = [g]$.
\end{lemma}

\begin{proof}
	We claim that any automorphism $\psi$ of $\G_{\iota_1}$ maps $V_1$ to itself. This is clear when $\pi_1(H) = G_1$. Thus we assume that $\pi_1(H)\ne G_1$ (which in particular implies that $| J_{1}^\ast| \geq 1)$. Notice that then, $s$ is the only vertex connected to itself through an edge $(s, s)\in R_i(\G_{\iota_1})$ of label $i \in J_{1}^\ast$. But  $\psi$ being a morphism of $I$-binary systems implies that $\big(\psi(s), \psi(s)\big)\in R_i(\G_{\iota_1})$ for $i \in J_{1}^\ast$,  which leads to $\psi(s) = s$ and our claim holds.

	Now, on the one hand, given $g\in \pi_1(H)$ it is immediate to check that we obtain an automorphism $\psi_g\in\Aut_{IRel}(\G_{\iota_1})$ of binary $I$-systems by declaring $\psi_g([h])=[h][g]=\phi_{[g]^{-1}}([h])$ for $[h]\in V_1$, and $\psi_g(s)=s$.

	On the other hand, given $\psi\in\Aut_{IRel}(\G_{\iota_1})$ such that $\psi([e_{G_1}]) = [g]$, and bearing in mind Remark \ref{fullsubGi1}, we can now affirm that  $\psi|_{V_1}$ is an automorphism of the full relational subsystem $\Cay\big(V_1,\{[r_i] \mid i\in I_{\iota_1}\sqcup J_{\pi_1}^\ast\}\big)$. Hence, $\psi|_{V_1} = \phi_{[g]^{-1}}$ (see Remark \ref{remark:fixvertex}), the only automorphism sending $[e_{G_1}]$ to $[g]$, and since $\psi(s)=s$, then $\psi=\psi_g$.
\end{proof}

To prove that $G_2 \cong \Aut_{\IRel}(\G_2)$, we first show that any element $\tilde{g}\in G_2$ induces an automorphism $\Phi_{\tilde{g}}$ on $\G_2$.  We now give the construction of  $\Phi_{\tilde{g}}$  and then we prove that it is indeed an automorphism of relational systems.

\begin{definition}\label{def:InducedAut}
	Given $\tilde g \in G_2$, we define $ \Phi_{\tilde{g}}: V(\G_2) = G_2\sqcup \big(\sqcup_{j\in J_2} V_2^{j}\big) \rightarrow V(\G_2) $ as follows. First, given that $\Cay(G_2,S)$ is a full relational subsystem of $\G_2$ (see Remark \ref{rem:fullsub})  we define ${\Phi_{\tilde{g}}}|_{G_2} $ as $ \phi_{\tilde{g}}$, the automorphism induced by right multiplication by ${\tilde g}^{-1}$ in  $\Cay(G_2,S)$. Thus for $g \in G_2$
		\[\Phi_{\tilde{g}} (g) = g \tilde g^{-1} \in G_2.\]
	Secondly, for $(j, [g])\in V_2^{j}$, we define
		\[\Phi_{\tilde{g}}(j,[g]) = \big(j_2(s_j{\tilde g}^{-1}), [g ]\theta^{-1}[k_2(s_j{\tilde g}^{-1})]\big)\in V_2^{j_2(s_j{\tilde g}^{-1})}.\]
	If moreover $\pi_1(H)\ne G_1$,  for $(j, s)\in V_2^{j}$, we finally define
		\[\Phi_{\tilde g}(j, s) = \big(j_2(s_j {\tilde g}^{-1}),s\big) \in V_2^{j_2(s_j{\tilde g}^{-1})}.\]
\end{definition}

The previous self-map of $V(\G_2)$ is indeed a homomorphism of the relational system $\G_2$:

\begin{lemma}\label{lem:InducedAut} 
	Given $\tilde g \in G_2$, $\Phi_{\tilde{g}} \in \Hom_{\IRel} (\G_2, \G_2)$.
\end{lemma}

\begin{proof}
	We check that $\Phi_{\tilde{g}}$ is a morphism of binary $I$-systems, that is, $\Phi_{\tilde{g}}$ respects relations $R_i(\G_2)$, $i\in I$. We prove it by cases:

	For $g \in G_2$, then if $i\in I_2$, we have that $(g, s_i g)\in R_i(\G_2)$ and
		\[\big(\Phi_{\tilde{g}}(g),\Phi_{\tilde{g}}(s_i g)\big) = (g \tilde{g}^{-1}, s_i g \tilde{g}^{-1}) \in R_i(\G_2).\]
	If $\theta \in I$,  we have that $\big(g,(j_2(g),\theta^{-1}[k_2(g)])\big)\in R_\theta(\G_2)$ and
	\begin{align*}
		\big(\Phi_{\tilde{g}}(g),\Phi_{\tilde{g}}(j_2(g),&\theta^{-1}[k_2(g)])\big)  = \\
		& = \big(g \tilde{g}^{-1}, (j_2(s_{j_2(g)}\tilde{g}^{-1}), \theta^{-1}[k_2(g)] \theta^{-1}[k_2(s_{j_2(g)}\tilde{g}^{-1})] )\big) \\
		&  = \big(g \tilde{g}^{-1}, (j_2(g \tilde{g}^{-1}), \theta^{-1}[k_2(g \tilde{g}^{-1})] )\big)\in R_\theta(\G_2).
	\end{align*}
	where the last equality follows from Lemma \ref{lemma:PropjAndk}\emph{(1)} and \emph{(2)},   and the fact that  $\theta^{-1}$ is a group homomorphism.
	
	For $g\in \pi_1(H)$, then if $i\in I_{i_1}$, we have $\big((j, [g]), (j, [g])\big)\in R_i(\G_2)$,  $j\in J_2$,  and
	\begin{align*} 
		\big(\Phi_{\tilde{g}}(j, [g]), &\Phi_{\tilde{g}}(j, [g])\big)  = \\ 
		& =\big((j_2(s_j\tilde{g}^{-1}), [g] \theta^{-1}[k_2(s_j\tilde{g}^{-1})]), (j_2(s_j\tilde{g}^{-1}), [g] \theta^{-1}[k_2(s_j\tilde{g}^{-1})] ) \big),
	\end{align*}
	which is an edge in $R_i(\G_2)$. If $ i\in J_{\pi_1}^\ast$, we have $\big((j,[g]),(j,[r_i g])\big)\in R_i(\G_2)$,  for $j\in J_2$,  and
	\begin{align*}
		\big(\Phi_{\tilde{g}}(j,[g]),&\Phi_{\tilde{g}}(j,[r_i g])\big) =  \\
		&=\big((j_2(s_j\tilde{g}^{-1}),[g] \theta^{-1}[k_2(s_j\tilde{g}^{-1})]), (j_2(s_j\tilde{g}^{-1}), [r_i g ]\theta^{-1}[k_2(s_j\tilde{g}^{-1})])\big),
	\end{align*}
	which is an edge in $R_i(\G_2)$.
	
	If moreover $\pi_1(H)\ne G_1 $, then for $i\in J_1 \setminus J_{\pi_1}$, we have $\big((j, s), (j, [g])\big)$ and $\big((j, [g]), (j, s)\big)$ in $R_i(\G_2)$, $j\in J_2$. As both are analogous, we only check the first:
	 	\[\big(\Phi_{\tilde{g}}(j, s), \Phi_{\tilde{g}}(j, [g])\big) = \big((j_2(s_j\tilde{g}^{-1}), s), (j_2(s_j\tilde{g}^{-1}), [g] \theta^{-1}[k_2(s_j \tilde{g}^{-1})])\big)  \in R_i(\G_2).\]
	For $i\in I_2$ then $\big((j, s),(j, s)\big)\in R_i(\G_2)$, $j\in J_2$, and
		\[\big(\Phi_{\tilde{g}}(j, s),\Phi_{\tilde{g}}(j, s)\big) = \big((j_2(s_j\tilde{g}^{-1}), s),(j_2(s_j\tilde{g}^{-1}), s)\big) \in R_i(\G_2).\]
\end{proof}

Indeed, the construction from Definition \ref{def:InducedAut} is an automorphism of the relational system and defines a group homomorphism $\Phi$ as follows:

\begin{proposition}\label{prop:GroupHom}
	The following hold:
	\begin{enumerate}
		\item Given $\tilde g \in G_2$, the morphism $\Phi_{\tilde g} \in \Aut_{\IRel}(\G_2)$.
		\item The following map is a group homomorphism
			\[\begin{array}{rcl}
				\Phi: G_2 & \longrightarrow & \Aut_{\IRel}(\G_2)\\
				\tilde{g} & \longmapsto & \Phi_{\tilde{g}}
			\end{array}\]
	\end{enumerate}
\end{proposition}

\begin{proof}
	We are going to prove that for $\tilde{g}, \tilde{h},  \tilde{g}\tilde{h}\in G_2$, we have that $\Phi_{\tilde{g}} \Phi_{\tilde{h}}  =  \Phi_{\tilde{g} \tilde{h}}.$ Indeed
	\begin{align*}
		\Phi_{\tilde{g}}\big(\Phi_{\tilde{h}}(j,[g])\big) &= \Phi_{\tilde{g}}\big(j_2(s_j{\tilde{h}}^{-1}), [g ]\theta^{-1}[k_2(s_j{\tilde{h}}^{-1})]\big) \\
		&= \big( j_2(s_{j_2(s_j\tilde{h}^{-1})}\tilde{g}^{-1}), [g] \theta^{-1}[k_2(s_j\tilde{h}^{-1})] \theta^{-1}[k_2(s_{j_2(s_j\tilde{h}^{-1})}\tilde{g}^{-1})] \big),
	\end{align*}
	and,
	\begin{align*}
		\Phi_{\tilde{g}\tilde{h}}(j,[g])\big) &=
		\big(j_2(s_j(\tilde{g}\tilde{h})^{-1}), [g]\theta^{-1}[k_2(s_j(\tilde{g}\tilde{h})^{-1})]\big) \\
		& = \big(j_2(s_j \tilde{h}^{-1}\tilde{g}^{-1}), [g]\theta^{-1}[k_2(s_j\tilde{h}^{-1}\tilde{g}^{-1})]\big).
	\end{align*}
	By Lemma \ref{lemma:PropjAndk}\emph{(1)}, $j_2(s_{j_2(s_j\tilde{h}^{-1})}\tilde{g}^{-1}) = j_2(s_j \tilde{h}^{-1}\tilde{g}^{-1})$, and by Lemma \ref{lemma:PropjAndk}\emph{(2)},
		\[k_2(s_j\tilde{h}^{-1}) k_2(s_{j_2(s_j\tilde{h}^{-1})}\tilde{g}^{-1}) = k_2(s_j \tilde{h}^{-1}\tilde{g}^{-1}).\]
	As $\theta^{-1}$ is a group isomorphism, it follows then that $\Phi_{\tilde{g}}\big(\Phi_{\tilde{h}}(j, [g])\big) = \Phi_{\tilde{g}\tilde{h}}(j,[g])$ for $(j, [g])\in V_2^{j}$.

	Finally, if $\pi_1(H)\ne G_1$, for $(j, s)\in V_2^{j}$ we have
	\begin{align*}
		\Phi_{\tilde{g}}\big(\Phi_{\tilde{h}}(j, s)\big) &= \Phi_{\tilde{g}}\big(j_2(s_j\tilde{h}^{-1}), s\big)\\
		&= \big(j_2(s_{j_2(s_j\tilde{h}^{-1})}\tilde{g}^{-1}), s\big) = \big(j_2(s_j \tilde{h}^{-1}\tilde{g}^{-1}), s\big) = \Phi_{\tilde{g}\tilde{h}}(j,s),
	\end{align*}
	as a consequence of Lemma \ref{lemma:PropjAndk}\emph{(1)}. Now, from Definition \ref{def:InducedAut} it is clear that $\Phi_{e_{G_2}}$ is the identity map of $\G_2$, and since $\Phi_{\tilde{g}^{-1}}\circ\Phi_{\tilde{g}}  = \Phi_{\tilde{g}^{-1}\tilde{g} } =\Phi_{e_{G_2}} $, we obtain that $\Phi_{\tilde{g}}$ and $\Phi_{\tilde{g}^{-1}}$ are inverse maps. Then, by Lemma \ref{lem:InducedAut}, \emph{(1)} is proved.
		
	Now \emph{(2)} follows directly from the fact that  $\Phi$ is then well-defined, and that we have just proved that $\Phi_{\tilde{g}} \Phi_{\tilde{h}}  =  \Phi_{\tilde{g} \tilde{h}}$ for $\tilde{g}, \tilde{h} \in G_2$.
\end{proof}

We have all the ingredients to show that $G_2 \cong \Aut_{\IRel}(\G_2)$.

\begin{lemma}\label{lem:GroupIso} 
	The morphism $\Phi: G_2  \rightarrow \Aut_{\IRel}(\G_2)$ from Proposition \ref{prop:GroupHom} is an isomorphism.
\end{lemma}

\begin{proof}
	It is  straightforward to show that $\Phi$ is a monomorphism since $\Phi_{\tilde{g}}(e_{G_2}) = \tilde{g}^{-1}$.

	To check that $\Phi$ is an epimorphism  we need to show that every automorphism $\psi\in\Aut_{\IRel}(\G_2)$ is equal to $\Phi_{\tilde{g}}$, from Definition \ref{def:InducedAut}, for some $\tilde{g}\in G_2$. Since $\psi$ is a morphism of $\G_2$, it must respect the edges $R_i(\G_2)$, $i\in I$, which in particular implies that $\psi (G_2)$ is  contained in $G_2$.  Moreover,
	$\left.\psi\right|_{\G_2(G_2)}$ is an automorphism of the full relational subsystem $\G_2(G_2)  = \Cay(G_2,S)$ that must be induced by an element $\tilde{g}\in G_2$. That is, the map $\left.\psi\right|_{\G_2(G_2)}$ is $\phi_{\tilde g}$ introduced in Remark \ref{remark:fixvertex}.  We claim that $\psi = \Phi_{\tilde{g}}$.
		
	By construction, we have that $\left.\psi\right|_{\G_2(G_2)} =  \left.\Phi_{\tilde{g}}\right|_{\G_2(G_2)}\in \Aut_{\IRel}(\Cay(G_2,S) )$. Now, the only edge in $R_\theta(\G_2)$ starting at $g \in G_2$ is $\big(g,(j_2(g),\theta^{-1}[k_2(g)])\big)$ and, since $\psi(g) = \Phi_{\tilde{g}}(g)$, we have that $\psi\big(j_2(g),\theta^{-1}[k_2(g)]\big) = \Phi_{\tilde{g}}\big(j_2(g),\theta^{-1}[k_2(g)]\big)$. In particular for $g = s_j$, $j\in J_2$ we get that $\psi(j,[e_{G_1}]) = \Phi_{\tilde{g}}(j,[e_{G_1}])$ and therefore for $j\in J_2$, $(\psi^{-1}\circ \Phi_{\tilde{g}})(j,[e_{G_1}])=(j,[e_{G_1}])$.  Using that composition must also preserve edges in $R_i(\G_2)$, for all $i\in I_1$, we obtain that $(\psi^{-1}\circ \Phi_{\tilde{g}})(V_2^{j})\subseteq V_2^{j}$, for all $j\in J_2$. In fact, $\psi^{-1}\circ\Phi_{\tilde{g}}$ induces an automorphism of the corresponding copy of $\G_{\iota_1}$. Then, by Lemma \ref{lemma:autGi1}, $\psi^{-1}\circ \Phi_{\tilde{g}}$ restricted to $V_2^{j}$ must be the identity,  $j \in J_2$, so we conclude that $\psi = \Phi_{\tilde{g}}$.
\end{proof}

We now have the necessary ingredients to prove our main theorem in Section \ref{section:realcolgraph}.

\subsection{Proof of Theorem \ref{theorem:MainArrowRealization}}\label{subsec:proofmaintheorem}

Consider the morphism $\varphi\colon \G_1\to \G_2$ from Definition \ref{def:phi}. As we have mentioned at the beginning of Subsection \ref{subsec:proprelsyst}, since $\G_1 $ is a Cayley diagram for $G_1$, we have that $\Aut_{\IRel}(\G_1)\cong G_1$.  Also, from Lemma \ref{lem:GroupIso}, we have that $\Aut_{\IRel}(\G_2)\cong G_2$. It only remains to show that $\Aut_{\IRel}(\varphi) \cong H$.

First consider $(\phi_{\tilde{g}_1},\Phi_{\tilde{g}_2}) \in \Aut_{\IRel}(\varphi)$, where  $\phi_{\tilde{g}_1}$ is
the automorphism of $\G_1 =\Cay(G_1,S)$ from Remark \ref{remark:fixvertex},  and $\Phi_{\tilde{g}_2}$ the automorphism of $\G_2$ constructed in Definition \ref{def:InducedAut} for $ \tilde{g}_2 \in  G_2$.  We are going to show that $(\tilde{g}_1, \tilde{g}_2) \in H$. Indeed, since $(\phi_{\tilde{g}_1},\Phi_{\tilde{g}_2}) \in \Aut_{\IRel}(\varphi)$, we have that  $ \Phi_{\tilde{g}_2}\circ \varphi =\varphi\circ \phi_{\tilde{g}_1}$.  Now, by construction
\begin{eqnarray*}
	\Phi_{\tilde{g}_2}\circ\varphi(e_{G_1}) &=& \Phi_{\tilde{g}_2}(0, [e_{G_1}]) = \big(j_2(\tilde{g}_2^{-1}), \theta^{-1}[k_2(\tilde{g}_2^{-1})]\big), \\
	\varphi\circ\phi_{\tilde{g}_1}(e_{G_1}) & =& \varphi(\tilde{g}_1^{-1})= \big(0, \varphi_0(\tilde{g}_1^{-1})\big),
\end{eqnarray*}
and therefore $j_2(\tilde{g}_2^{-1}) = 0$. So  $\tilde{g}_2^{-1}\in \pi_2(H)$ and $k_2(\tilde{g}_2^{-1}) = \tilde{g}_2^{-1}$, which from the previous equations, leads us to $\varphi_0(\tilde{g}_1^{-1}) = \theta^{-1}[\tilde{g}_2^{-1}]$. This implies that $\tilde{g}_1^{-1}\in \pi_1(H)$ and moreover, $\theta([\tilde{g}_1^{-1}]) = [\tilde{g}_2^{-1}]$. Hence, by Lemma \ref{lemma:goursat}, we obtain that $(\tilde{g}_1^{-1}, \tilde{g}_2^{-1}) \in H$, and therefore $(\tilde{g}_1,\tilde{g}_2)\in H$.
	
Conversely, we prove that for $(\tilde{g}_1, \tilde{g}_2)\in H$, the couple $(\phi_{\tilde{g}_1},\Phi_{\tilde{g}_2})\in \Aut_{\IRel}(\G_1) \times  \Aut_{\IRel}(\G_2)$ is indeed an element in $\Aut_{\IRel}(\varphi)$. For that, we need to show that $(\varphi\circ \phi_{\tilde{g}_1})(g) = (\Phi_{\tilde{g}_2}\circ \varphi)(g)$, for every $g\in G_1$. First observe that since $(\tilde{g}_1, \tilde{g}_2)\in H$, we have that $(\tilde{g}_1^{-1}, \tilde{g}_2^{-1})\in H$ and therefore $\theta^{-1}[\tilde{g}_2^{-1}] = [\tilde{g}_1^{-1}]$.
Now, on the one hand,
	\[\varphi\circ \phi_{\tilde{g}_1}(g) = \varphi(g \tilde{g}_1^{-1}) = \big(0, \varphi_0(g \tilde{g}_1^{-1})\big) = 
	\begin{cases}
		(0, [g \tilde{g}_1^{-1}]), & \text{if $g\in \pi_1(H)$},\\
		(0, s), & \text{otherwise}.
	\end{cases}\]
On the other hand, if $g\in \pi_1(H)$,
	\[\Phi_{\tilde{g}_2}\circ\varphi(g) = \Phi_{\tilde{g}_2}(0, [g]) = (0, [g] \theta^{-1}[\tilde{g}_2^{-1}]) = (0, [g \tilde{g}_1^{-1}]),\]
and, if $g\not \in \pi_1(H)$, then $\Phi_{\tilde{g}_2}\circ\varphi(g) = \Phi_{\tilde{g}_2}(0, s) = (0, s)$. The result thus follows. $\hfill \square$

\section{Realisability in the arrow category of simple graphs}\label{section:toGraphs}

Our aim here is to prove Theorem \ref{theorem:arrowGraphRealisation} stated in the Introduction. We proceed by following the ideas developed by  Frucht \cite{Fru39}  and by de Groot \cite{Gro59}. Roughly speaking, in the previous section we introduced  binary $I$-systems  in the same spirit as Cayley diagrams were used in Frucht-de Groot's work to solve the classical group realisability problem. In this section an asymmetric graph (that is, a graph without non-trivial automorphisms) is assigned to each label of the binary $I$-systems from Theorem \ref{theorem:MainArrowRealization}, in such a way that these asymmetric graphs are pairwise non-isomorphic. Finally, by a process called replacement operation \cite[\S 4.4]{HelNes04},   every labelled edge is substituted by its corresponding asymmetric graph, thus obtaining finally simple undirected graphs. The vertex degrees should be carefully chosen so that any automorphism of the graph would map each asymmetric graph to a copy of itself. This is a key fact in the proof of Theorem \ref{theorem:arrowGraphRealisation}.

We first give a definition of  vertex degree which is compatible with the replacement operation, in the sense that the degree of a vertex in a relational system coincides with the degree of the same vertex in the final graph; then, we compute the vertex degrees in the binary $I$-systems from Definition \ref{def:G1andG2}; finally,  we prove Theorem \ref{theorem:arrowGraphRealisation} by performing the replacement operation in the binary $I$-systems from Theorem \ref{question:ArrowRealisation}.
	
\begin{definition}\label{def:degreeIRel}
	Let $\G$ be a binary $I$-system. For $v \in V(\G)$ we define:
	\begin{itemize}
		\item[--]  the \emph{indegree} of $v \in V(\G)$ as $\deg^-(v) = |\sqcup_{i\in I}\{w\in V(\G) \mid (w,v)\in R_i(\G)\}|$;
		\item[--]  the \emph{outdegree} of $v$ as $\deg^+(v) = |\sqcup_{i\in I}\{w\in V(\G) \mid (v, w)\in R_i(\G)\}|$;
		\item[--]  the \emph{degree} of $v$ as $\deg(v) = \deg^+(v) + \deg^-(v)$.
	\end{itemize}
	Observe that an edge $(v, v)\in R_i(\G)$ counts twice in $\deg(v)$.
\end{definition}

\begin{remark}\label{rem:degreeCayley}
	In a Cayley diagram $\G = \Cay(G, \{s_i\mid i\in I\})$, for each vertex $g\in G$ and for each generator $s_i$, there is exactly one edge labelled $i$ starting at $g$, $(g, s_i g)$, and one edge labelled $i$ ending at $g$, $(s_i^{-1}g, g)$. Thus, every vertex $v\in V(\G)$ has $\deg^+(v) = \deg^-(v) = |I|$, and $\deg(v) = 2|I|$.
\end{remark}
	
We next compute the vertex degrees in the binary $I$-systems from Definition \ref{def:G1andG2}.

\begin{lemma}\label{lemma:degreesVertices}
	Let $\G_1$ and $\G_2$ be the binary $I$-systems  from Definition \ref{def:G1andG2}. Then:
	\begin{enumerate}
		\item Vertices in $\G_1$ have degree $2|I_1|$;
		\item Vertices  in $\G_2$ have the following degree:	
		\begin{enumerate}	
			\item for $g_2 \in G_2$, $\deg(g_2) =2|I_2|+1$;
			\item for $(j, [g_1])\in V_2^{j}$, $g_1\in \pi_1(H)$, $\deg\big((j, [g_1])\big) =2|I_1| + |\iota_2^{-1}(H)|$;
			\item if a vertex $(j, s)\in V_2^{j}$ exists, $\deg(j, s) = \deg(s)\ge 2|I_1|$, where $s\in V(\G_{\iota_1})$.
		\end{enumerate}
	\end{enumerate}
\end{lemma}

\begin{proof}
	As $\G_1 = \Cay(G_1, R)$,  \emph{(1)} is straightforward from Remark \ref{rem:degreeCayley}. To prove \emph{(2)(a)} recall that $\G_2(G_2)$ is a full binary relational subsystem of $\G_2$ isomorphic to $\Cay(G_2,S)$ (see Remark \ref{rem:fullsub}) and that there exists a unique edge in $R_\theta(\G_2)$ starting at $g_2$. Therefore $\deg^+(g_2) = |I_2| + 1, \, \deg^-(g_2) = |I_2|$ and $\deg(g_2) = 2|I_2|+1$.
		
	To prove \emph{(2)(b)}, let $(j, [g_1]) \in V_2^{j}$, for some $g_1\in \pi_1(H)$.
	Recall that the full binary relational subsystem of $\G_2$ with vertices $V_2^{j}$ and edges with labels in $I_1$ is isomorphic to $\G_{\iota_1}$. If $G_1 = \pi_1(H)$,  $\G_{\iota_1}$ is isomorphic to $\Cay\big(V_1,\{[r_i] \mid i\in I_{\iota_1}\sqcup J_{\pi_1}^\ast\}\big)$. Hence, as in this case $J_{\pi_1} = J_1$, there are $|I_{\iota_1}\sqcup J_{\pi_1}^\ast| = |I_1|$  edges with labels in $I_1$ both starting and arriving at $(j, [g_1])$. If $G_1\ne \pi_1(H)$,  we also have to consider the edges $\big((j,[g_1]), (j, s)\big)$ and $\big((j, s), (j, [g_1])\big)$ for every label in $J_1 \setminus J_{\pi_1}$. Thus, there are a total of $|I_{\iota_1}\sqcup J_{\pi_1}^\ast| + |J_1 \setminus J_{\pi_1}| = |I_1|$ edges with labels in $I_1$ both arriving and ending at $(j, [g_1])$. Since no other edges start at $(j, [g_1])$, we obtain that $\deg^+\big((j,[g_1])\big) = |I_1|$. 	To compute the indegree of $(j,[g_1])$ we still have to check how many edges labelled $\theta$ arrive at $(j,[g_1])$.  Recall that edges in $R_\theta(\G_2)$ are of the form $\big(g, (j_2(g), \theta^{-1}[k_2(g)]) \big)$, $g \in G_2$.  Notice that the uniqueness of the decomposition $g_2 = s_{j_2(g_2)} k_2(g_2)$ (see Remark \ref{rem:decompgroups}), implies that any pair $(j, g)$, $j\in J_2$, $g\in \pi_2(H)$, appears exactly once as $\big(j_2(g_2), k_2(g_2)\big)$ for some $g_2\in G_2$. Then,  there are as many such edges arriving at $(j, [g_1])$ as elements $g_2\in\pi_2(H)$ verifying that $\theta^{-1}[g_2] = [g_1]$. Equivalently, there are as many edges labelled $\theta$ arriving at $(j, [g_1])$ as elements in the class of $\theta(g_1)$, hence there are $|\iota_2^{-1}(H)|$ such edges. Therefore, $\deg^-\big((j, [g_1])\big) = |I_1| + |\iota_2^{-1}(H)|$ and $\deg\big((j, [g_1])\big) = 2|I_1| + |\iota_2^{-1}(H)|$, proving \emph{(2)(b)}.
		
	Finally, the degrees of vertices $(j, s)\in V_2^{j}$ are not entirely determined. However, these vertices only take part in $R_i(\G_2)$, $i\in I_1$, and as we mentioned above, for every $j\in J_2$, the binary relational subsystem with vertices $V_2^{j}$ and edges with labels in $I_1$, is isomorphic to $\G_{i_1}$. Hence $\deg(j,s) = \deg(s)$, for $s\in V(\G_{i_1})$ and,  as $(s, s)\in R_i(\G_{i_1})$ for every $i\in I_1$, $\deg(s)\ge 2|I_1|$ and \emph{(2)(b)} follows.
\end{proof}

\subsection{Proof of Theorem \ref{theorem:arrowGraphRealisation}}

We finish the section with the somewhat lengthy proof of Theorem \ref{theorem:arrowGraphRealisation}. Let $G_1$, $G_2$ be groups and $H\le G_1\times G_2$.  By Theorem \ref{theorem:MainArrowRealization}, there exist two
binary $I$-systems $\G_1'$ and $\G_2'$, introduced in Definition \ref{def:G1andG2},  and a morphism $\varphi'\colon \G_1'\to \G_2'$ such that $\Aut_\IRel(\G_i')\cong G_i$, $i=1,2$, and $\Aut_{\IRel}(\varphi')\cong H$. We can assume that the set of labels $I = I_1\sqcup I_2\sqcup \{\theta\}$ satisfies that $I_1$ and $I_2$ have more than one element. Otherwise, we add to $I_i$ as many labels associated to $e_{G_i}$ as necessary, $i=1,2$.  Hence every vertex in  $\G_1'$ and $\G_2'$ has at least degree four as a consequence of Lemma \ref{lemma:degreesVertices}.
		
Let $\alpha = \max\{2|I_1|, \; 2|I_2|+1,\; 2|I_1|+|i_2^{-1}(H)|, \; \deg(s)\}$, $s\in V(\G_{i_1})$. For  each $i\in I$,  an asymmetric simple graph $R_i$ can be constructed verifying that, apart from a vertex of degree one, every other vertex has degree strictly greater than $\alpha$ (see \cite[Section 6]{Gro59}). Moreover, if $i,j\in I$, $i\ne j$, then $R_i$ and $R_j$ are non isomorphic.
	
The proof of Theorem \ref{theorem:arrowGraphRealisation} reduces to performing the replacement operation (see \cite[Section 6]{Gro59}) to the binary relational systems $\G_j'$, $j=1,2$. Explicitly, for every $i\in I$ and $(v, w)\in R_i(\G_j')$, consider a graph isomorphic to $R_i$, $R_i^{(v, w)}$, and denote its vertex of degree one by $p_i^{(v, w)}$. Then:

\begin{definition}\label{def:graph}
	Let $\G_j$, $j=1,2$, be the simple graph with vertices and edges:
	\begin{align*}
		V(\G_j) & = V(\G_j') \sqcup \big(\sqcup_{i\in I}\big(\sqcup_{(v, w)\in R_i(\G_j')}(V(R_i^{(v, w)})\sqcup \{r_i^{(v, w)}\})\big)\big),\\
		E(\G_j)& = \sqcup_{i\in I}\big(\sqcup_{(v, w)\in R_i(\G_j')}\big(E(R_i^{(v, w)})\sqcup\{(v, r_i^{(v, w)}), (r_i^{(v,w)}, p_i^{(v,w)}), (p_i^{(v,w)},w)\}\big) \big).
	\end{align*}
\end{definition}

The automorphism groups of $\G_j$ and $\G'_j$ coincide:

\begin{lemma}\label{lem:groups_coincide}
	For $j=1,2$, $\Aut_{\Graphs}(\G_j)\cong \Aut_{\IRel}(\G_j') = G_j$. Moreover, any automorphism of $\G_j$ is completely determined by how it acts on $V(\G_j')\subseteq V(\G_j)$.
\end{lemma}

\begin{proof}
	Observe that the degree of $v\in V(\G_j')\subseteq V(\G_j)$ is the same in both $\G_j'$ and $\G_j$.  Indeed, for each $(v,w)\in R_i(\G_j')$, there is an edge $(v, r_i^{(v,w)})\in E(\G_j)$, and for each $(w,v)\in R_i(\G_j')$,  there is another edge $(p_i^{(w, v)},v)\in E(\G_j)$. Given that these are the only edges in $\G_j$ incident to $v$, our claim holds. The vertices $r_i^{(v, w)}$ and $p_i^{(v, w)}$ have degree two and three respectively, while the remaining vertices in each of the $R_i^{(v, w)}$, $(v,w)\in R_i(\G_2')$, have the same degree as in $R_i$, thus greater than $\alpha$.

	We can now check that if $\psi\in \Aut_{\Graphs}(\G_j)$ then   $\psi|_{V(\G_j')}\in \Aut_{\IRel}(\G_j')$. Since an automorphism of graphs respects the degree of vertices, previous considerations on the degrees imply that $\psi$ restricts to a bijective map $\psi|_{V(\G_j')}\colon V(\G_j')\to V(\G_j')$. Now, for $(v,w)\in R_i(\G_j')$, we have that $(v, r_i^{(v, w)})\in E(\G_j)$, thus $\big(\psi(v), \psi(r_i^{(v, w)})\big)\in E(\G_j)$. Given that $\psi$ respects the degree of vertices, $\psi(r_i^{(v, w)}) = r_j^{(\psi(v),u)}$, for some vertex $u\in V(\G_j')$ and $j\in I$. For the same reason, $\psi$ restricts to an isomorphism $R_i^{(v,w)}\to R_k^{(\psi(v), w)}$. But this implies that $i = k$, so $u = \psi(w)$ and $\big(\psi(v),\psi(w)\big)\in R_i(\G_j')$. Therefore, $\psi|_{V(\G_j')}\in \Aut_{\IRel}(\G_j')$.
		
	Reciprocally, given $\psi'\in \Aut_{\IRel}(\G_j')$, we can naturally define a map $\psi\colon V(\G_j)\to V(\G_j)$ as follows. A vertex $v\in V(\G_j')$ is taken to $\psi(v) = \psi'(v)$ and, for $(v,w)\in R_i(\G_j)$, $i\in I$, define $\psi(r_i^{(v,w)}) = r_i^{(\psi'(v),\psi'(w))}$. Finally, define $\psi|_{R_i^{(v,w)}}\colon R_i^{(v,w)}\to R_i^{(\psi'(v),\psi'(w))}$ as the identity map between the two copies of $R_i$. Then it is clear that $\psi\in \Aut_{\Graphs}(\G_j)$.
\end{proof}

Finally, given the morphism of binary relational systems $\varphi'\colon \G_1' \to \G_2'$, we can naturally define a morphism of graphs $\varphi\colon \G_1\to \G_2$ as follows. The vertex $v\in V(\G_1')\subseteq V(\G_1)$ is taken to $\varphi(v) = \varphi'(v)\in V(\G'_2)\subseteq V(\G_2)$, the vertex $r_i^{(v, w)}$ is taken to $r_i^{(\varphi'(v), \varphi'(w))}$, and $\varphi$ restricts to the identity map $R_i^{(v,w)}\to R_i^{(\varphi'(v),\varphi'(w))}$. Clearly, $\varphi$ is a morphism of graphs. To conclude the proof of Theorem \ref{theorem:arrowGraphRealisation}, it remains to see that:

\begin{lemma}\label{lem:aut_varphi}
	$\Aut_{\Graphs}(\varphi) = \Aut_{\IRel}(\varphi') = H$.
\end{lemma}

\begin{proof}
 	Consider  $(\phi'_{\tilde{g}_1}, \Phi'_{\tilde{g}_2} ) \in \Aut_{\IRel}(\G'_1) \times \Aut_{\IRel}(\G'_2)$, $i=1,2$, where $\phi'_{\tilde{g}_1}$ is the automorphism of $\G'_1 =\Cay(G_1, S)$ introduced in Remark \ref{remark:fixvertex} and $\Phi'_{\tilde{g}_2}$ the automorphism of $\G'_2$ constructed in Definition \ref{def:InducedAut} for $ \tilde{g}_2 \in  G_2$.  Let $(\phi_{\tilde{g}_1}, \Phi_{\tilde{g}_2} ) \in \Aut_{\IRel}(\G_1) \times \Aut_{\IRel}(\G_2)$  be the morphism whose components extend, respectively,   $\phi'_{\tilde{g}_1}$ and $\Phi'_{\tilde{g}_2} $.  Then $\varphi \circ \phi_{\tilde{g}_1}|_{V(\G'_1)} = \varphi' \circ \phi'_{\tilde{g}_1}$ and $\Phi_{\tilde{g}_2}\circ \varphi|_{V(\G_1')} = \Phi'_{\tilde{g}_2} \circ \varphi'$. Since automorphisms of $\G_i$ are uniquely determined by how they act on $V(\G_i')$, $i=1,2$, we obtain that $\varphi \circ \phi_{\tilde{g}_1} = \Phi_{\tilde{g}_2}\circ \varphi$ if and only if $\varphi' \circ \phi'_{\tilde{g}_1} = \Phi'_{\tilde{g}_2} \circ \varphi'$, that is, $(\phi_{\tilde{g}_1}, \Phi_{\tilde{g}_2})\in \Aut_{\Graphs}(\varphi)$ if and only if $(\phi'_{\tilde{g}_1}, \Phi'_{\tilde{g}_2})\in \Aut_{\IRel}(\varphi')$. The result follows.
\end{proof}
	
Notice that the asymmetric graphs $R_i$ constructed in \cite[Section 6]{Gro59} are infinite. Hence, the graphs $\G_j$ in Definition \ref{def:graph}, and thus in Theorem \ref{theorem:arrowGraphRealisation}, will also be infinite. Nevertheless, as long as the groups involved are finite, there is a way  of obtaining finite graphs:

\begin{corollary}\label{corollary:finiteGraphs}
	Let $G_1$, $G_2$ be finite groups and $H\le G_1\times G_2$. There exist $\G_1$, $\G_2$ finite objects in $\Graphs$ and $\varphi\colon \G_1\to \G_2$ object in $\Arr(\Graphs)$ such that $\Aut_{\Graphs}(\G_i) \cong G_i$, $i=1,2$ and $\Aut_{\Graphs}(\varphi) \cong H$.
\end{corollary}

\begin{proof}
	In \cite[Section 1]{Fru39} Frucht shows that there exist infinitely many finite asymmetric graphs that can be used for an arrow replacement such that the highest of the degrees of their vertices is three. Indeed one can choose any starlike tree $T$ whose root $v\in V(T)$ has degree $3$ and such that the length of the three paths of $T-{v}$ differ, see \cite[Fig.\ 1]{Fru39}. If $G_1$ and $G_2$ are finite, we can ensure that $|I_1|$ and $|I_2|$ are both greater than one, and by considering a family of $|I|$ finite asymmetric graphs, we obtain the result by following the lines of the proof of Theorem \ref{theorem:arrowGraphRealisation}.
\end{proof}


\section{Realisability in the arrow categories of \texorpdfstring{$\CDGA$}{CDGA} and \texorpdfstring{$\mathcal{H}oTop_*$}{HoTop}}\label{section:topSpaces}

In Theorem \ref{theorem:arrowGraphRealisation} we gave a positive answer to Question \ref{question:ArrowRealisation} in $ \Arr(\Graphs)$.  In this section we are going to transfer this result to $\Arr(\CDGA)$, Theorem \ref{theorem:CDGArealisation},  via a functor that preserves automorphism groups.  Indeed, we construct a family of functors which are the restriction to the category of $\Graphs$ of a family of functors $\mathcal M_n \colon \Dig \rightarrow \CDGA$,  $n \in \mathbb N$, introduced in Definition \ref{definition:CDGAfunctor}. Recall that $\Dig$ denotes the category of strongly connected digraphs (i.e. directed graphs) with more than one vertex, and that a digraph $\G$ is strongly connected if for every pair of vertices $v,w\in V(\G)$, there exists an integer $m \in {\mathbb N}$ and vertices $v = v_0, v_1,\dots, v_m = w$ such that $(v_{i-1},v_i)\in E(\G)$, $i=1,2,\dots,m$. Since any simple graph can be seen as a symmetric digraph, \cite[\S 1.1]{HelNes04}, if it is connected, then the associated digraph is strongly connected.  So, by abuse of notation, the restriction of the previous functor will also be denoted by $\mathcal M_n \colon \Graphs \rightarrow \CDGA$.

Once that the realisability result is settled in $\Arr(\CDGA)$, we use Sullivan's spatial realisation functor to obtain a positive answer to Question \ref{question:ArrowRealisation} when the groups that are involved are finite, Corollary \ref{theorem:HoTopRealisation}.

\subsection{Families of elliptic CDGA's}\label{sub:ellipticCDGA}

We follow a similar approach as in \cite{CosMenVir18, CosVir14} where, to every finite simple graph $\G$,  a minimal Sullivan algebra $M_\G$ is assigned in such a way that the group of self-homotopy equivalences of $M_\G$ is isomorphic to the automorphism group of $\G$. Our construction in \cite[Definition 2.1]{CosVir14} was based on a {homotopically rigid} algebra given in \cite{AL2}, and it was functorial only when restricted to the subcategory of full graph monomorphisms (see  \cite[Remark 2.8]{CosVir14}).  However, the morphism $\varphi$ obtained in Theorem \ref{theorem:arrowGraphRealisation} (see also Corollary \ref{corollary:finiteGraphs}) is not a full monomorphism in general, so our previous construction is useless to answer Question \ref{question:ArrowRealisation} in $\Arr(\CDGA)$.  In this work we provide a new family of minimal Sullivan algebras, Definition \ref{definition:dgamodel},  that leads to a well defined functor in Subsection \ref{sub:functors}.  The main difference with \cite[Definition 2.1]{CosVir14} is that our new minimal Sullivan algebras are based on (strictly) rigid algebras (see Remark \ref{rem:Mnrigid}) introduced in \cite[Definition 1.1]{CosMenVir18}.  And, the main difference with \cite[Definition 2.1]{CosMenVir18}  is that now they have generators in every edge of the graph, whereas previously the edges where only codified by the differential.   These two differences altogether imply that the group of automorphisms and the group of self-homotopy equivalences coincide (see Corollary \ref{cor:selfEquiv}.\eqref{corollary:cdgaRealisation}), which is crucial to us.

\begin{definition}\label{definition:dgamodel}
	Let $\mathcal{G}$ be a strongly connected digraph with more than one vertex. For each $n\geq 1$, we associate to $\mathcal{G}$ the minimal Sullivan algebra
		\[\M_n(\G) =\Big(\Lambda(x_1,x_2,y_1,y_2,y_3,z)\otimes \Lambda\big(x_v,  z_{(v, w)} \mid v\in V(\G), (v, w)\in E(\G)\big),d\Big)\]
	where
		\[\begin{array}{ll}
			|x_1|=30n+18,\hspace{1cm} & d x_1 = 0,\\
			|x_2|=36n+22, & d x_2 = 0,\\
			|y_1|=126n+75, & d y_1 = x_1^3 x_2,\\
			|y_2|=132n+79, & d y_2 = x_1^2 x_2^2,\\
			|y_3|=138n+83, & d y_3 = x_1 x_2^3,\\
			|x_v|=180n^2+218n+66, & d x_v = 0,\\
			|z|=540n^2+654n+197, & d z = x_1^{18n}(x_2^2 y_1 y_2 - x_1 x_2 y_1 y_3 + x_1^2 y_2 y_3) \\& \hspace{13.5pt} + x_1^{18n+11} + x_2^{15n+9},\\
			|z_{(v, w)}|=540n^2+654n+197,\qquad & d z_{(v, w)} = x_v^3 + x_v x_w x_2^{5n+3} + x_1^{18n+11}.
		\end{array}\]
\end{definition}

Observe that $\M_n(\G)$ is $(30n+17)$-connected. These algebras have further desirable properties:

\begin{proposition}
	Let $\G$ be a finite strongly connected digraph with more than one vertex. For each $n\ge 1$, the minimal Sullivan algebra $\M_n(\G)$ is elliptic.
\end{proposition}

\begin{proof}
	We have to prove that the cohomology of $\M_n(\G)=(\Lambda V,d)$ is finite dimensional, which is equivalent to proving that the cohomology of the pure Sullivan algebra associated to $(\Lambda V,d)$ is finite dimensional, \cite[Proposition 32.4]{FelHalTho01}. This pure algebra is $(\Lambda V,d_\sigma)$, where the differential $d_\sigma$ is defined as
		\[\begin{array}{lll}
			d_\sigma(x_1) = 0,\qquad & d_\sigma(y_1) = x_1^3 x_2,\qquad & d_\sigma(z) = x_1^{18n+11} + x_2^{15n+9}, \\
			d_\sigma(x_2) = 0, & d_\sigma(y_2) = x_1^2 x_2^2, & d_\sigma(z_{(v, w)}) = x_v^3 + x_v x_w x_2^{5n+3} + x_1^{18n+11}, \\
			d_\sigma(x_v) = 0, & d_\sigma(y_3) = x_1 x_2^3.
		\end{array}\]
	
	To prove that the cohomology of $(\Lambda V, d_\sigma)$ is finite dimensional it suffices to show that  the powers of the cohomological classes $[x_1]$, $[x_2]$ and $[x_v]$, $v\in V(\G)$ eventually vanish. Indeed,
		\[d_\sigma(z x_1 - y_3 x_2^{15n+6}) = x_1^{18n+12}, \qquad d_\sigma(z x_2 - y_1 x_1^{18n+8}) = x_2^{15n+10}.\]
	Moreover, given $v\in V(\G)$, the strong connectivity of $\G$ implies that $v$ is the starting vertex of at least one edge $(v, w)\in E(\G)$, and from $d z_{(v, w)}$ we obtain that
		\[[x_v^3]^4 = [- x_v x_w x_2^{5n+3} - x_1^{18n+11}]^4 = 0.\]
	This proves the ellipticity of the algebra.
\end{proof}

Building on \cite{CosMenVir18}, we now introduce some lemmas that are needed for the proof of Theorem \ref{theorem:CDGArealisation}, which will be carried out in Subsection \ref{sub:proofmaintheorem}. The next result deals with the degrees of elements in the algebras introduced in Definition \ref{definition:dgamodel}, and extends \cite[Lemma 2.5]{CosMenVir18}.
\begin{lemma}\label{lemma:fullyisolated}
	Let $\G$ be a strongly connected digraph with more than one vertex and let $n\ge 1$ be an integer. Then:
	\begin{enumerate}
		\item A basis of $\M_n(\G)^{|x_v|}$ is $\{x_2^{5n+3},\, x_v\mid v\in V(\G)\}$.
		\item A basis of $\M_n(\G)^{|z|}$ is $\{z, \,z_{(v,w)}\mid (v,w)\in E(\G)\}$.
	\end{enumerate}
\end{lemma}

\begin{proof}
	Elements of degree  $|x_v|$, other than $x_v, v\in V(\G)$, are of the form $P$, $P_{12} y_1 y_2$, $P_{13} y_1 y_3$ and $P_{23} y_2 y_3$, where $P, P_{12}, P_{13}, P_{23}\in \Q[x_1, x_2]$. In view of the proof of \cite[Lemma 2.5]{CosMenVir18}, $P$ is a multiple of $x_2^{5n+3}$. Hence,  to prove (1), it suffices to show that $P_{12}$, $P_{13}$ and $P_{23}$ are trivial. Let $m\in \{|x_v|-|y_1 y_2|, |x_v| - |y_1 y_3|, |x_v| - |y_2 y_3|\}$ and let us show that there is no pair of non-negative integers $(\alpha, \beta)$ which are a solution of the diophantine equation $m = \alpha|x_1| + \beta|x_2|.$ That way, no monomial $x_1^\alpha x_2^\beta$ has degree $m$ and therefore \emph{(1)} follows.  Following the proof of \cite[Lemma 2.5]{CosMenVir18}, by choosing suitable particular solutions for the previous diophantine equation, we obtain its  general solution:

	\begin{center}
		\begin{tabular}{c|c}
		 $m$ & general solution \\[1pt] \hline
			$|x_v| - |y_1 y_2|$ & $\big(-11 + k(18n + 11), 5n + 5 - k(15n + 9)\big)$ \\[2pt]
			$|x_v| - |y_1 y_3|$ & $\big(-10 + k(18n + 11), 5n + 4 - k(15n + 9)\big)$ \\[2pt]
			$|x_v| - |y_2 y_3|$ & $\big(-9 + k(18n + 11), 5n + 3 - k(15n + 9)\big)$ \\[2pt]
		\end{tabular}
	\end{center}
	It is clear that for  $k > 0$, the second component is negative, and for $k \leq 0$, the first component is negative. Thus there is no solution where both integers are non-negative, and \emph{(1)} follows.
	
	To prove \emph{(2)} we follow the same approach. By \cite[Lemma 2.5]{CosMenVir18}, no multiple of $y_1 y_2 y_3$ has degree $|z|$. We now consider the products of $y_1$, $y_2$ or $y_3$ with polynomials on the generators of even order and prove that there is no pair of non-negative integers $(\alpha, \beta)$ such that $m = \alpha|x_1| + \beta|x_2|$, with $m\in \{|z|-|y_j|, |z|-|y_j| - |x_v|, |z| - |y_j| - 2|x_v|\}$, for $j=1,2,3$. As previously, we obtain the following general solution:

	\begin{center}
		\begin{tabular}{c|c}
		 $m$ & general solution \\[1pt]\hline
			$|z| - |y_1|$ & $\big(-3 + k(18n + 11), 15n + 8 - k(15n + 9)\big)$ \\[2pt]
			$|z| - |y_2|$ & $\big(-2 + k(18n+11), 15n + 7 - k(15n + 9))$\\[2pt]
			$|z| - |y_3|$ & $\big(-1 + k(18n+11), 15n + 6 - k(15n + 9))$\\[2pt]
			$|z| - |y_1| - |x_v|$ & $\big(-3 + k(18n + 11), 10n + 5 - k(15n + 9)\big)$ \\[2pt]
			$|z| - |y_2| - |x_v|$ & $\big(-2 + k(18n + 11), 10n + 4 - k(15n + 9)\big)$ \\[2pt]
			$|z| - |y_3| - |x_v|$ & $\big(-1 + k(18n + 11), 10n + 3 - k(15n + 9)\big)$ \\[2pt]
			$|z| - |y_1| - 2|x_v|$ & $\big(-3 + k(18n + 11), 5n + 2 - k(15n + 9)\big)$ \\[2pt]
			$|z| - |y_2| - 2|x_v|$ & $\big(-2 + k(18n + 11), 5n + 1 - k(15n + 9)\big)$ \\[2pt]
			$|z| - |y_3| - 2|x_v|$ & $\big(-1 + k(18n + 11), 5n - k(15n + 9)\big)$ \\[2pt]
		\end{tabular}
	\end{center}
	Again, it is clear that the first coordinate is non-negative if and only if $k>0$, in which case $\beta$ is negative. Thus \emph{(2)} follows.	
\end{proof}

\subsection{Families of functors}\label{sub:functors}
Henceforward, by abuse of notation, we will use the same letters $x_1$, $x_2$, $y_1$, $y_2$, $y_3$ and $z$, as it will be clear from the context whether we work in $\M_n(\G_1)$ or $\M_n(\G_2)$.

\begin{lemma}\label{lemma:inducedMorphism}
	Let $\G_1$ and $\G_2$ be strongly connected digraphs with more than one vertex and let $n\ge 1$ be an integer. Every $\sigma\in \Hom_{\Dig}(\G_1,\G_2)$ induces a morphism of Sullivan algebras $\M_n(\sigma)\colon \M_n(\G_1)\to \M_n(\G_2)$. Moreover if $\M_n(\sigma) =\M_n(\sigma')$ for some $\sigma,\sigma' \in \Hom_{\Dig}(\G_1,\G_2)$, then $\sigma = \sigma'$.
\end{lemma}

\begin{proof}
	Since $\sigma\in \Hom_{\Dig}(\G_1,\G_2)$, given $(v, w)\in E(\G_1)$, $\big(\sigma(v),\sigma(w)\big)\in E(\G_2)$. We can then define $\M_n(\sigma)\colon \M_n(\G_1)\to \M_n(\G_2)$ as follows:
	\[\begin{array}{ll}
	\M_n(\sigma)(w)=w, & \text{for $w\in \{x_1, x_2, y_1, y_2, y_3, z\}$},\\
	\M_n(\sigma)(x_v)=x_{\sigma(v)}, & \text{for all $v\in V(\G_1)$},\\
	\M_n(\sigma)(z_{(v, w)}) = z_{(\sigma(v), \sigma(w))}, & \text{for all $(v, w)\in E(\G_1)$}.
	\end{array}\]
	Simple computations show that $d \M_n(\sigma) = \M_n(\sigma)d$.
	
	Now, for $\sigma' \neq \sigma $,  there exists $v \in V(\G_1) $ such that $\sigma' (v) \neq \sigma (v)$. Hence, $x_{\sigma(v)} \neq x_{\sigma'(v)}$  and we conclude that
	$\M_n(\sigma) \neq \M_n(\sigma')$ as we wanted.
\end{proof}

\begin{definition}\label{definition:CDGAfunctor} (Family of functors $\mathcal M_n: \Dig \rightarrow \CDGA$)
	For every integer $n \geq 1$, we construct a functor $\M_n$,  from the category of strongly connected digraphs with more than one vertex, to the category of $n$-connected (indeed $(30n+17)$-connected) differential graded algebras as follows. To an object $\G$,  the Sullivan algebra $\M_n(\G)$ from Definition \ref{definition:dgamodel} is associated, and to a morphism $\sigma\in \Hom_{\Dig}(\G_1,\G_2)$,  the morphism $\M_n(\sigma)  \in \Hom_{\CDGA}\big(\M_n(\G_1), \M_n(\G_2)\big)$ from Lemma \ref{lemma:inducedMorphism} is associated.
\end{definition}

\begin{lemma}\label{lemma:computingMorphisms}
	Let $\G_1$ and $\G_2$ be strongly connected digraphs with more than one vertex and let $n\ge 1$ be an integer.  Let $f \colon \M_n(\G_1) \rightarrow \M_n(\G_2)$ be a morphism of $\CDGA$'s. Then, there exists $s \in \{0, 1\}$ such that $f (x_1) = sx_1, \, f (x_2) = sx_2, \, f (y_1) = sy_1,  \, f (y_2) = sy_2, \, f (y_3) = sy_3, \, f (z) = sz$.
\end{lemma}

\begin{proof}
	Let $f\in \Hom_{\CDGA}\big(\M_n(\G_1),\M_n(\G_2)\big)$. For degree reasons (see \cite[Lemma 1.3]{CosMenVir18}), $f(x_1)=a_1 x_1$, $f(x_2)=a_2 x_2$, $f(y_1) = b_1 y_1$, $f(y_2)= b_2 y_2$ and $ f(y_3) = b_3 y_3$, $a_1, a_2, b_1, b_2, b_3 \in \Q$. And, as $d f = f d$, we immediately obtain
	\begin{equation}\label{eq:relab}
		b_1 = a_1^3 a_2, \quad b_2 = a_1^2 a_2^2, \quad b_3 = a_1 a_2^3.
	\end{equation}
	As a consequence of Lemma \ref{lemma:fullyisolated},
	\begin{equation}\label{eq:fz}
		f(z) = c z + \sum_{(r, s)\in E(\G_2)} c(r, s) z_{(r, s)},
	\end{equation}

	Now, $d f(z) = f (d z)$. By \eqref{eq:fz},
	\begin{equation}\label{eq:dfz}
		\begin{split}	
			d f(z) &= c\big[x_1^{18n}(x_2^2 y_1 y_2 - x_1 x_2 y_1 y_3 + x_1^2 y_2 y_3) + x_1^{18n+11} + x_2^{15n+9} \big] \\
			&+ \sum_{(r, s)\in E(\G_2)} c(r, s) \big[ x_r^3 + x_r x_s x_2^{5n+3} + x_1^{18n+11}\big],
		\end{split}
	\end{equation}
	and on the other hand,
	\begin{equation}\label{eq:fdz}
		\begin{split}
			f(d z) =&\ a_1^{18n} x_1^{18n}[a_2^2 b_1 b_2 x_2^2 y_1 y_2 - a_1 a_2 b_1 b_3 x_1 x_2 y_1 y_3 + a_1^2 b_2 b_3 x_1^2 y_2 y_3] \\+&\ a_1^{18n+11} x_1^{18n+11} + a_2^{15n+9} x_2^{15n+9}.
		\end{split}
	\end{equation}
	Comparing equations \eqref{eq:dfz} and \eqref{eq:fdz}, we immediately see that $c(r, s) = 0$, for all $(r, s)\in E(\G_2)$. We also obtain the following identities.
	\begin{equation}\label{eq:relsz}
		c = a_1^{18n+11} = a_2^{15n+9} = b_1 b_2 a_1^{18n} a_2^2 = b_1 b_3 a_1^{18n+1} a_2 = b_2 b_3 a_1^{18n+2}.
	\end{equation}
	Equations \eqref{eq:relab} and \eqref{eq:relsz} are the same as in \cite[Lemma 2.12]{CosMenVir18}, thus there exists $s\in \{0,1\}$ such that $s = a_1 = a_2 = b_1 = b_2 = b_3 = c$ and we conclude.
\end{proof}

\begin{remark}[{\cite[Definition 1.1]{CosMenVir18}}]\label{rem:Mnrigid} 
	For each $n\geq 1$, we let $M_n =\Big(\Lambda(x_1,x_2,y_1,y_2,y_3,z), d\Big)$ be the subalgebra of the minimal Sullivan algebra $\mathcal M_n (\mathcal G)$ introduced in Definition \ref{definition:dgamodel}. Following exactly the same arguments as in the proof of the previous lemma, it is immediate to prove that  $M_n$ is a rigid algebra, that is  $\Hom_{\CDGA}\big(M_n, M_n\big) = \{0,1\}.$
\end{remark}

\begin{lemma}\label{lem:szero} 
	Under the assumptions of Lemma \ref{lemma:computingMorphisms}, $f$ is the trivial  homomorphism if and only if $s=0$.
\end{lemma}

\begin{proof}
	The first implication is trivial. Assume now that $s = 0$, so $f (x_1) = f (x_2)= f (y_1) =f (y_2)= f (y_3) =f (z)= 0.$ As a consequence of Lemma \ref{lemma:fullyisolated},
	\begin{equation}\label{eq:fxv}
		f(x_v) = \sum_{r\in V(\G_2)}  a(v, r) x_r + a(v) x_2^{5n+3}, \quad v\in V(\G_1),
	\end{equation}
	\begin{equation}\label{eq:fzvw}
		f(z_{(v, w)}) = e(v, w) z + \sum_{(r, s)\in E(\G_2)} c\big((v, w),(r, s)\big) z_{(r, s)}, \quad (v, w) \in E(\G_1),
	\end{equation}
	where almost every coefficient $a(v,r), \, c(r,s)$ and $c\big((v,w), (r,s)\big)$ is zero.

	We know nothing about the constants involved in $f(x_v)$ and $f(z_{(v, w)})$ so far. To get some information, we use that $f(d z_{(v, w)}) = d f(z_{(v, w)})$. On the one hand, by \eqref{eq:fzvw},
	\begin{align}\label{eq:dfzvw}
		\begin{split}
			df(z_{(v, w)}) &= e(v, w)\big[x_1^{18n}(x_2^2 y_1 y_2 - x_1 x_2 y_1 y_3 + x_1^2 y_2 y_3) + x_1^{18n+11} + x_2^{15n+9} \big]\\
			&+ \sum_{(r, s)\in E(\G_2)} c\big((v, w),(r, s)\big) \big[ x_r^3 + x_r x_s x_2^{5n+3} + x_1^{18n+11}\big].
		\end{split}
	\end{align}
	On the other hand, we have
	\begin{align}\label{eq:fdcase0}
		\begin{split}
			f d(z_{(v, w)}) &=  f (x_v^3 + x_v x_w x_2^{5n+3} + x_1^{18n+11}) = f(x_v^3)\\
			& =\left[\sum_{r\in V(\G_2)} a(v, r) x_r  + a(v) x_2^{5n+3} \right]^3
		\end{split}
	\end{align}
	as we are assuming that $s= 0$ and so $f (x_1) = f (x_2) =0$.

	Comparing equations \eqref{eq:dfzvw} and \eqref{eq:fdcase0}, we immediately obtain that $e(v, w) = 0$. Then, a summand containing $x_2^{15n+9}$ does not appear in \eqref{eq:dfzvw}, which implies that $a(v) = 0$. But this implies that summands containing $x_r x_s x_2^{5n+3}$ do not appear in \eqref{eq:fdcase0}. Comparing with \eqref{eq:dfzvw}, we see that $c\big((v, w),(r, s)\big) = 0$, for all $(r, s)\in E(\G_2)$. This also implies that $a(v, r)=0$, for all $r\in V(\G_2)$. Thus $f(x_v) = 0$, for all $v\in V(\G_1)$, and $f(z_{(v,w)})=0$, for all $(v,w)\in E(\G_1)$. In other words $f$ is the trivial morphism and we conclude the proof.
\end{proof}

The next result shows that our  family of functors is suitable for our purposes as they are \emph{almost} fully faithful functors:

\begin{theorem}\label{theorem:computingMorphisms}	
	For any integer $n\geq 1$, the functor $\M_n: \Dig \rightarrow \CDGA$ induces the following bijective correspondence:
		\[\Hom_{\Dig}(\G_1,\G_2)\cong\Hom_{\CDGA}\big(\M_n(\G_1),\M_n(\G_2)\big)^\ast.\]	
\end{theorem}

\begin{proof}
	Let $f \in \Hom_{\CDGA}\big(\M_n(\G_1),\M_n(\G_2)\big)^\ast$.  We claim that $f = \M_n(\sigma)$, for some $\sigma\in \Hom_{\Dig}(\G_1,\G_2)$. Observe that, since $f $ is not the trivial homomorphism, by the two previous lemmas, Lemma \ref{lemma:computingMorphisms} and Lemma \ref{lem:szero}, we get that $s =1.$
	
	Notice also that the strong connectivity of the graph implies that for every $v\in V(\G_1)$, $v$ is the starting vertex of an edge $(v, w)\in E(\G_1)$. Therefore the coefficients $a(v, r)$ and $a(v)$ involved in $f(x_v)$ (see \eqref{eq:fxv})  can be entirely determined by using that $f(d z_{(v, w)}) = d f(z_{(v, w)})$.  So, on the one hand, $d f(z_{(v, w)})$ is as in \eqref{eq:dfzvw}, whereas
	\begin{align}\label{eq:fdzvw}
		\begin{split}
			f &d(z_{(v, w)}) = \Bigg[\sum_{r\in V(\G_2)} a(v, r) x_r + a(v) x_2^{5n+3} \Bigg]^3 + x_1^{18n+11}\\
			&+ \Bigg(\sum_{r\in V(\G_2)} a(v, r) x_r + a(v) x_2^{5n+3}\Bigg)\Bigg(\sum_{s\in V(\G_2)} a(w, s) x_s + a(w) x_2^{5n+3}\Bigg)x_2^{5n+3}.
		\end{split}
	\end{align}

	First, no coefficient $x_1^{18n+2} y_2 y_3$ exists in \eqref{eq:fdzvw}. Thus, $e(v, w)=0$. Now, no summand containing $x_u x_v x_w$ exists in \eqref{eq:dfzvw}, if $u\ne v\ne w\ne u$. Such a summand would appear in \eqref{eq:fdzvw} if there were three or more non-trivial coefficients $a(v, r)$. We can then assume that there are at most two non-trivial $a(v, r)$. But if there were two non-trivial coefficients, summands containing $x_r x_s^2$ would appear in \eqref{eq:fdzvw}. Since they do not appear in \eqref{eq:dfzvw}, for each $v\in V(\G_1)$, there is at most one non-trivial coefficient $a(v, r)$. Consequently, for each $(v, w)\in E(\G_1)$, there is also at most one non-trivial coefficient $c\big((v, w),(r, s)\big)$.
	
	Suppose that $c\big((v, w),(r, s)\big) = 0$, for every $(r, s)\in E(\G_2)$. Then, $dfz_{(u,w)}=0=fd(z_{(u,w)})$. But a summand $x_1^{18n+11}$ appears in \eqref{eq:fdzvw}. Therefore, there exists a unique $(r, s)\in E(\G_2)$ such that $c\big((v, w),(r, s)\big) \ne 0$. Hence we have that $a(v, r)\ne 0$, that is, for every vertex $v \in V(\G_1)$ there is exactly one non-trivial coefficient $a(v, r)$.
	
	Denote by $\sigma(v)$ the only vertex in $V(\G_2)$ such that $a\big(v,\sigma(v)\big)\ne 0$.We define a map $\sigma \colon V(\G_1) \to V(\G_2)$ that takes $v\in V(\G_1)$ to $\sigma(v) \in V(\G_2)$. We claim that $\sigma$ is a morphism of graphs. Indeed if $(v, w)\in E(\G_1)$, the non-trivial coefficient $c\big((v, w),(r, s)\big)$ verifies, by comparing equations \eqref{eq:dfzvw} and \eqref{eq:fdzvw}, that $r = \sigma(v)$ and $s = \sigma(w)$. This exhibits that $\big(\sigma(v),\sigma(w)\big)\in E(\G_2)$, so $\sigma\in \Hom_{\Graphs}(\G_1,\G_2)$. Furthermore, comparing the coefficient of $x_1^{18n+11}$ in \eqref{eq:dfzvw} and \eqref{eq:fdzvw}, we obtain that $c\big((v, w),(\sigma(v),\sigma(w)\big) = 1$. Then, comparing the coefficients of $x_{\sigma(v)}^3$, we see that $a\big(v,\sigma(v)\big) = 1$, for all $v\in V(\G_1)$. Finally, notice that there are no summands $x_{\sigma(v)}^2 x_2^{5n+3}$ in \eqref{eq:dfzvw}. They would appear in \eqref{eq:fdzvw} if $a(v)\ne 0$, thus we deduce that $a(v)=0$, for all $v\in V(\G_1)$.
	
	Then, we have proved that there exists $\sigma\in \Hom_{\Dig}(\G_1,\G_2)$ such that for every $v\in V(\G_1)$, $f(x_v) = x_{\sigma(v)}$, and for every $(v,w)\in E(\G_1)$, $f(z_{(v,w)}) = z_{(\sigma(v),\sigma(w))}$. Since $s=1$, this implies that $f = \M_n(\sigma)$ as we claimed. Finally, to conclude we use Lemma \ref{lemma:inducedMorphism}.	
\end{proof}

The following results are immediate consequences of Theorem \ref{theorem:computingMorphisms}.

\begin{corollary}\label{theorem:correspondence}
	Let $\G_1$ and $\G_2$ be objects in $\Dig$. Then, for $n\geq 1$,  $$\Hom_{\CDGA}\big(\M_n(\G_1),\M_n(\G_2)\big) = [\M_n(\G_1), \M_n(\G_2)].$$
\end{corollary}	

\begin{proof}\label{rem:trivialhomotopy}
	The strategy to prove this result follows the same outline as in \cite{CosVir14, CosVir18}. From our previous result, the homotopy equivalence relation is trivial in $\Hom_{\CDGA}\big(\M_n(\G_1), \M_n(\G_2)\big)$. Indeed, for different elements $\sigma, \tau\in\Hom_{\Dig}(\G_1,\G_2)$, the induced morphisms $\M_n(\sigma),\M_n(\tau)$ have different linear parts. Hence,  by  \cite[Proposition 12.8]{FelHalTho01},  they are not homotopic.  Notice also that, by the same argument, the trivial morphism in $\Hom_{\CDGA}\big(\M_n(\G_1),\M_n(\G_2)\big)$ is not homotopic to $\M_n(\sigma)$ for any  $\sigma \in \Hom_{\Dig}(\G_1,\G_2)$ since the linear part of the former one is never trivial.
\end{proof}

Recall now that any group can be represented as the automorphism group of a graph, \cite{Gro59}. We obtain the following:

\begin{corollary}\label{cor:selfEquiv}For every integer $n \geq 1 $, we have the following:
	
	\begin{enumerate}
		\item  For $\G$ an object in $\Dig$,
			\[\Aut_{\Dig}(\G) \cong \Aut_{\CDGA}\big(\M_n (\G)\big) = \E\big(\M_n(\G)\big).\]
		\item \label{corollary:cdgaRealisation} For $G$ a group, there exists an $n$-connected object $M_n$ in $\CDGA $ such that
			\[G \cong \Aut_{\CDGA}(M_n) = \E\big(M_n\big).\]
	\end{enumerate}
\end{corollary}

\subsection{Proof of Theorem \ref{theorem:CDGArealisation}}\label{sub:proofmaintheorem} 
We now have all the ingredients to derive Theorem \ref{theorem:CDGArealisation}, which gives a positive answer to Question \ref{question:ArrowRealisation}  in $\mathcal C = \CDGA$. By Theorem \ref{theorem:arrowGraphRealisation}, there exist objects $\G_1$, $\G_2$ in $\Graphs$ and  $\psi\colon \G_1\to \G_2$ object in $\Arr(\Graphs)$, such that $\Aut_{\Graphs}(\G_i) \cong G_i$, for $i=1,2$, and $\Aut_{\Graphs}(\psi) \cong H$. Now we consider $\mathcal M_n: \Graphs \rightarrow \CDGA$ the restriction of the functor from Definition \ref{definition:CDGAfunctor} and the corresponding objects $M_i = \M_n(\G_i)$, $i=1,2$  in $\CDGA$ and $\varphi = \M_n(\psi) \colon M_1 \to  M_2$ in $\Arr(\CDGA)$. By Corollary \ref{cor:selfEquiv}, we obtain that $\Aut_{\CDGA}(M_i)\cong \Aut_{\Dig}(\G_i)\cong G_i$, $i=1,2$.

We now prove that $ \Aut_{\Graphs}(\psi) \cong \Aut_{\CDGA}(\varphi)$. First, consider $(\sigma_1,\sigma_2)\in \Aut_{\Graphs}(\psi)$. Then, as $\psi\circ\sigma_1 = \sigma_2\circ \psi$, by functoriality we also have that  $\varphi\circ \M_n(\sigma_1) = \M_n(\sigma_2)\circ \varphi$. Moreover,  as $\M_n(\sigma_i) \in \Aut_{\CDGA}(M_i)$, we deduce that $\big(\M_n(\sigma_1), \M_n (\sigma_2)\big)\in \Aut_{\CDGA}(\varphi)$.

Reciprocally, consider $(f_1, f_2)\in \Aut_{\CDGA}(\varphi)$.  Then, as $f_i$  is an automorphism of $M_i$, $i=1,2$,   by Theorem \ref{theorem:computingMorphisms} and Corollary \ref{cor:selfEquiv},  there exist $\sigma_i\in \Aut_{\Graphs}(\G_i)$ such that $f_i = \M_n (\sigma_i)$, $i=1,2$.  Now, as $\big(\M_n (\sigma_1), \M_n (\sigma_2)\big)\in \Aut_{\CDGA}(\varphi)$,  we have that $\M_n (\sigma_2)\circ\varphi = \varphi\circ \M_n (\sigma_1)$.
That is, for every $v\in V(\G_1)$:
	\[\big(\M_n (\sigma_2)\circ\varphi\big) (x_v) =x_{(\sigma_2\circ\psi)(v)} =x_{(\psi\circ \sigma_1)(v)} =  \big(\varphi\circ \M_n (\sigma_1)\big)(x_v). \]
Hence, ${(\sigma_2\circ\psi)(v)} = {(\psi\circ\sigma_1)(v)}$ for every $v\in V(\G_1)$, so $(\sigma_1, \sigma_2)\in \Aut_{\Graphs}(\psi)$. Then, $\Aut_{\CDGA}(\varphi) = \big\{\big(\M_n({\sigma_1}), \M_n({\sigma_2})\big)\mid (\sigma_1,\sigma_2)\in \Aut_{\Graphs}(\psi)\big\}\cong H$.
\hfill $\square$

\subsection{Proof of Corollary \ref{theorem:HoTopRealisation}}
We gather previous results in order to settle in the positive Question \ref{question:ArrowRealisation}  when $\mathcal C = \mathcal{H}oTop_*$ and the groups involved are finite.  Recall that  the spatial realisation functor $| \quad |: \CDGA \rightarrow  \mathcal{H}oTop_*$, \cite[Chapter 17]{FelHalTho01}, gives a bijective correspondence between rational homotopy types of simply connected spaces of finite type and isomorphism classes of finite type minimal Sullivan algebras, and that homotopy classes of continuous maps of rational spaces are in bijection to the homotopy classes of morphisms of Sullivan algebras.

Now, for $G_1$, $G_2$ finite groups and $H\le G_1\times G_2$, using Corollary \ref{corollary:finiteGraphs},   there exist $\G_1$, $\G_2$ finite objects in $\Graphs$ and a morphism $\psi \in \Hom_{\Graphs}(\G_1,\G_2)$ such that $\Aut_{\Graphs}(\G_i) \cong G_i$, $i=1,2$ and $\Aut_{\Graphs}(\psi) \cong H$. By Theorem \ref{theorem:CDGArealisation}, $M_i = \mathcal M_n (\G_i)$ is an $n$-connected (finite type) object in $\CDGA$ such that $\Aut_{\CDGA}(M_i)\cong G_i$, $i = 1, 2$, and  $\varphi= \mathcal M_n (\psi)$ is a morphism verifying $\Aut_{\CDGA}(\varphi) \cong H$. Observe that since the graphs are finite, the corresponding algebras are of finite type. Also, as we mentioned previously, the homotopy relation is trivial in these algebras, so it is immediate that $\E(M_i)\cong G_i$, $i=1,2$, and $\E(\varphi)\cong H$. Then, by letting  $X_i = |M_i|$, $i=1,2$, and $f = |\varphi|\colon X_1 \to X_2$,  and using the properties of the realisation functor, the proof of Corollary \ref{theorem:HoTopRealisation} is straightforward. \hfill $\square$

\subsection{Proof of Corollary \ref{theorem:categoryRepresentation}}
For any concrete small category $\mathcal{C}$ there exists a fully faithful functor $G\colon\mathcal{C}\to \Dig$,  \cite[Chapter 4, 1.11]{PulTrn80}. Thus, if we define $G_n = \M_n \circ G \colon \mathcal{C} \to \CDGA $,  by Theorem \ref{theorem:computingMorphisms} we obtain that
	\[\Hom_{\CDGA}\big(G_n(A),G_n(B)\big)^\ast = \Hom_{\Dig}\big(G(A), G(B)\big) = \Hom_\mathcal{C}(A,B).\]
	
We now turn to topological spaces, thus assume that $\mathcal{C}$ is finite. Clearly, $\mathcal{C}^\text{op}$ is finite as well. Then, by \cite[Corollary 4.26]{HelNes04}, there exists a fully faithful functor $F$ from $\mathcal{C}^\text{op}$ to the category of (finite) connected graphs, which in particular is a subcategory of $\Dig$. By regarding $F$ as a contravariant functor, we can define $F_n = |\quad|\circ \M_n \circ F  \colon  \mathcal{C} \to \mathcal HoTop_\ast$ which takes any object of $\mathcal{C}$ into a Sullivan algebra of finite type. Using the properties of the spatial realisation functor and applying Theorem \ref{theorem:computingMorphisms} and Corollary \ref{theorem:correspondence}, $\big[F_n(A), F_n(B)\big]^\ast = [\M_n(F(B)), \M_n(F(A))]^\ast = \Hom_{\Graphs}\big(F(B), F(A)\big) = \Hom_\mathcal{C}(A,B)$ for any $A,B \in \mathcal{C}$. \hfill $\square$

\section{Example} \label{sec:example}

Let us illustrate the constructions involved in the proof of Theorem \ref{theorem:MainArrowRealization} with an example. Let $G_1 = \Z_8$, $G_2 = \Z_4$ and let $H \leq G_1\times G_2$ be the subgroup generated by the element $(2,2) \in G_1\times G_2$, namely, $H = \{(0,0), (2,2), (4,0), (6,2)\}$. We are now going to describe the isomorphism $\theta$ from Lemma \ref{lemma:goursat}.

	First, $\pi_1(H) = \langle 2 \rangle \le \Z_8$, and $\iota_1^{-1}(H) = \langle 4\rangle \le \Z_8$. The quotient $\pi_1(H)/\iota_1^{-1}(H)$ is isomorphic to $\Z_2$, containing the classes $[0] = \{0,4\}$ and $[2] = \{2,6\}$. In a similar way, $\pi_2(H)=\langle 2\rangle\le \Z_4$, and $\iota_2^{-1}(H)$ is the trivial group, so the quotient $\pi_2(H)/\iota_2^{-1}(H) \cong \pi_2(H)$ contains the classes $[0] = \{0\}$ and $[2] = \{2\}$.  Hence, we have:
	\begin{align*}
		\theta\colon \frac{\pi_1(H)}{\iota_1^{-1}(H)}& \longrightarrow \frac{\pi_2(H)}{\iota_2^{-1}(H)},\\
		[0] = \{0,4\} & \longmapsto [0] = \{0\},\\
		[2] = \{2,6\} & \longmapsto [2] = \{2\}.
	\end{align*}

	The generating sets $R$ and $S$ for $G_1$ and $G_2$ respectively (see Definition \ref{def:decompgroups}) are described here below:
	
\begin{enumerate}
	
	\item	 There are four right cosets of $\iota_1^{-1}(H)$ in $G_1$. Let $J_1 = \{0,1,2,3\}$ and  denote $r_j = j\in G_1$, $j\in J_1$. Then the set $\{r_j\mid j\in J_1\}$ contains a representative of each right coset of $\iota_1^{-1}(H)$ in $G_1$. Moreover, if we denote $I_{\iota_1}=\{4\}$ and $r_4 = 4\in G_1$, $\{r_i\mid i\in I_{\iota_1}\}$ is a generating set for $\iota_1^{-1}(H)$. Taking $I_1 = I_{\iota_1}\sqcup J_1^* = \{1,2,3,4\}$, the set $R = \{r_i\mid i\in I_1\} = \{1,2,3,4\}$ is our generating set for $G_1$. We also need to consider the set $J_{\pi_1} = \{j\in J_1\mid r_j\in \pi_1(H)\} = \{0,2\}$ introduced in Remark \ref{rem:decompgroups}.

	\item	 There are two right cosets of $\pi_2(H)$ in $G_2$, so take $J_2 = \{0,1\}$. If we denote $s_j = j\in G_2$, $\{s_j\mid j\in J_2\}$ contains a representative of each of the two right cosets. Set $I_{\pi_2} = \{2\}$ and $s_2 = 2\in G_2$, so $\{s_i\mid i\in I_{\pi_2}\}$ is a generating set for $\pi_2(H)$. Taking $I_2 = I_{\pi_2}\sqcup J_2^* = \{1, 2\}$, the set $S = \{s_i\mid i\in I_2\} = \{1,2\}$ is our generating set for $G_2$.
	
	\end{enumerate}
	Recall now, from Remark \ref{rem:decompgroups}, the maps $k_1\colon G_1\to\iota_1^{-1}(H)$, $j_1\colon G_1\to J_1$, $k_2\colon G_2\to\pi_2(H)$ and $j_2\colon G_2\to J_2$. We  have then the necessary ingredients to build the binary relational systems solving Question \ref{question:ArrowRealisation} in our case.
	
	Let $I = I_1\sqcup I_2\sqcup \{\theta\}$. We first build the auxiliary $I$-system $\G_{\iota_1}$ introduced in Definition \ref{def:auxsystem}.
Since $\pi_1(H)\ne G_1$, the set of vertices is $V(\G_{\iota_1}) = V_1\sqcup\{s\}$, where $V_1 = \pi_1(H)/\iota_1^{-1}(H) = \{[0], [2]\}$, and the labelled edges are shown in Figure \ref{figure:Gi1}. Note that in the following diagrams, a two-headed arrow means that there is an edge of the corresponding label in each  direction.
	\begin{center}
		\begin{figure}[H]
			\begin{tikzpicture}
				\tikzset{node distance = 2.5cm, auto}
				\node[style=point,label={[label distance=0cm]180:$[2]$}](1) {};
				\node[style=point,below of = 1,label={[label distance=0cm]180:$[0]$}](2) {};
				\node[right of = 1,style=point,xshift=-0.6cm,yshift=-1.25cm, label={[label distance=8pt]0:$s$}](3) {};
				\draw[<->,thick,>=stealth] (1) to node {} (2);
				\draw[<->,thick,>=stealth, color = red] (2) to[in=210+15,out=30-15] node {} (3);
				\draw[<->,thick,>=stealth, color = blue] (2) to[in=210-20,out=30+20] node {} (3);
				\draw[<->,thick,>=stealth, color = blue] (1) to[in=150+20,out=330-20] node {} (3);
				\draw[<->,thick,>=stealth, color = red] (1) to[in=150-15,out=330+15] node {} (3);
				\draw[->,thick,>=stealth, green,looseness=9] (1) to[out = 125, in = 55] node {} (1);
				\draw[->,thick,>=stealth, green,looseness=9] (2) to[out = 235, in = 305] node {} (2);
				\draw[->,thick,>=stealth, green,looseness=9] (3) to[out = 80, in = 10] node {} (3);
				\draw[->,thick,>=stealth, blue,looseness=9] (3) to[out = -80, in = -10] node {} (3);
				\draw[->,thick,>=stealth, red,looseness=7] (3) to[in = 65, out = 25] node {} (3);
				\draw[->,thick,>=stealth,looseness=7] (3) to[in = -65, out = -25] node {} (3);
			\end{tikzpicture}
			\caption{$\G_{\iota_1}$}\label{figure:Gi1}
		\end{figure}
	\end{center}
	
According to Definition \ref{def:G1andG2}, $\G_1=\Cay(G_1,R)$. Using the same colours as in Figure \ref{figure:Gi1} to represent labels, we get:
	\begin{center}
		\begin{figure}[H]
			\begin{tikzpicture}
				\tikzset{node distance = 2cm, auto}
				\node[style=point,label={[label distance=0cm]270:$0$}](0) {};
				\node[style=point,right of = 0, label={[label distance=0cm]270:$1$}](1) {};
				\node[style=point, above of = 1, xshift = 1.414cm, yshift = -0.586cm, label={[label distance=0cm]0:$2$}](2) {};
				\node[style=point, above of = 2, label={[label distance=0cm]0:$3$}](3) {};
				\node[style=point, above of = 3, xshift = -1.414cm, yshift = -0.586cm, label={[label distance=0cm]90:$4$}](4) {};
				\node[style=point, left of = 4, label={[label distance=0cm]90:$5$}](5) {};
				\node[style=point, below of = 5, xshift = -1.414cm, yshift = 0.586cm, label={[label distance=0cm]180:$6$}](6) {};
				\node[style=point, below of = 6, label={[label distance=0cm]180:$7$}](7) {};
				\foreach \x/\y in {0/1,1/2,2/3,3/4,4/5,5/6,6/7,7/0}
					\draw[->, thick, color = blue, >=stealth] (\x) to node{} (\y);
				\foreach \x/\y in {0/2,1/3,2/4,3/5,4/6,5/7,6/0,7/1}
					\draw[->, thick, >=stealth] (\x) to node{} (\y);
				\foreach \x/\y in {0/3,1/4,2/5,3/6,4/7,5/0,6/1,7/2}
					\draw[->, thick, color = red, >=stealth] (\x) to node{} (\y);
				\foreach \x/\y in {0/4,1/5,2/6,3/7}
					\draw[<->, thick, color = green, >=stealth] (\x) to node {} (\y);
			\end{tikzpicture}
			\caption{$\G_1 = \Cay(G_1,R)$}
		\end{figure}
	\end{center}
	
Finally, $\G_2$ has $\Cay(G_2,S)$ as a full binary $I$-subsystem  and two copies of $\G_{\iota_1}$ (as many as elements in $J_2$). Moreover, it has edges labelled by $\theta$ starting at each vertex in $\Cay(G_2,S)$. The binary relational system $\G_2$ is then as follows:
\begin{center}
		\begin{figure}[H]
			\begin{tikzpicture}
				\tikzset{node distance = 2.5cm, auto}
				\node[style=point,label={[label distance=0cm]170:$(0,[2])$}](1a) {};
				\node[style=point,below of = 1a,label={[label distance=0cm]190:$(0,[0])$}](2a) {};
				\node[style=point,left of = 1a,xshift=0.6cm,yshift=-1.25cm, label={[label distance=8pt]180:$(0,s)$}](3a) {};
				\draw[<->,thick,>=stealth] (1a) to node {} (2a);
				\draw[<->,thick,>=stealth, color = red] (2a) to[in=315,out=165] node {} (3a);
				\draw[<->,thick,>=stealth, color = blue] (2a) to[in=-10,out=130] node {} (3a);
				\draw[<->,thick,>=stealth, color = blue] (1a) to[in=10,out=230] node {} (3a);
				\draw[<->,thick,>=stealth, color = red] (1a) to[in=45,out=195] node {} (3a);
				\draw[->,thick,>=stealth, green,looseness=9] (1a) to[out = 125, in = 55] node {} (1a);
				\draw[->,thick,>=stealth, green,looseness=9] (2a) to[out = 235, in = 305] node {} (2a);
				\draw[->,thick,>=stealth, green,looseness=9] (3a) to[out = 100, in = 170] node {} (3a);
				\draw[->,thick,>=stealth, blue,looseness=9] (3a) to[out = 260, in = 190] node {} (3a);
				\draw[->,thick,>=stealth, red,looseness=7] (3a) to[in = 115, out = 155] node {} (3a);
				\draw[->,thick,>=stealth,looseness=7] (3a) to[in = 245, out = 205] node {} (3a);

				\node[style=point, right of = 1a, xshift = 0.5cm, label={[label distance=0cm]135:$3$}](3) {};
				\node[style=point, below of = 3, label={[label distance=0cm]225:$0$}](0) {};
				\node[style=point, right of = 0, label={[label distance=0cm]325:$1$}](1) {};
				\node[style=point, right of = 3, label={[label distance=0cm]45:$2$}](2) {};

				\foreach \x/\y in {0/1,1/2,2/3,3/0}
					\draw[->, thick, color = purple, >=stealth] (\x) to node{} (\y);
				\foreach \x/\y in {0/2,1/3}
					\draw[<->, thick, color=orange, >=stealth] (\x) to node{} (\y);

				\node[style=point, right of = 2,xshift = 0.5cm,label={[label distance=0cm]10:$(1,[2])$}](1b) {};
				\node[style=point,below of = 1b,label={[label distance=0cm]-10:$(1,[0])$}](2b) {};
				\node[style=point,right of = 1b, xshift=-0.4cm,yshift=-1.25cm, label={[label distance=8pt]0:$(1,s)$}](3b) {};
				\draw[<->,thick,>=stealth] (1b) to node {} (2b);
				\draw[<->,thick,>=stealth, color = red] (2b) to[in=210+15,out=30-15] node {} (3b);
				\draw[<->,thick,>=stealth, color = blue] (2b) to[in=210-20,out=30+20] node {} (3b);
				\draw[<->,thick,>=stealth, color = blue] (1b) to[in=150+20,out=330-20] node {} (3b);
				\draw[<->,thick,>=stealth, color = red] (1b) to[in=150-15,out=330+15] node {} (3b);
				\draw[->,thick,>=stealth, green,looseness=9] (1b) to[out = 125, in = 55] node {} (1b);
				\draw[->,thick,>=stealth, green,looseness=9] (2b) to[out = 235, in = 305] node {} (2b);
				\draw[->,thick,>=stealth, green,looseness=9] (3b) to[out = 80, in = 10] node {} (3b);
				\draw[->,thick,>=stealth, blue,looseness=9] (3b) to[out = -80, in = -10] node {} (3b);
				\draw[->,thick,>=stealth, red,looseness=7] (3b) to[in = 65, out = 25] node {} (3b);
				\draw[->,thick,>=stealth,looseness=7] (3b) to[in = -65, out = -25] node {} (3b);

				\draw[->, thick, >=stealth, color=yellow] (0) to node{} (2a);
				\draw[->, thick, >=stealth, color=yellow] (1) to node{} (2b);
				\draw[->, thick, >=stealth, color=yellow] (2) to[out = 210, in = -30] node{} (1a);
				\draw[->, thick, >=stealth, color=yellow] (3) to[out = -30, in = 210] node{} (1b);
			\end{tikzpicture}
			\caption{$\G_2$}\label{figure:G2}
		\end{figure}
	\end{center}

	The morphism of binary $I$-systems $\varphi\colon\G_1\to\G_2$ is easily obtained from Definition \ref{def:phi}.

\end{document}